\documentclass[11pt]{amsproc}
  \usepackage[margin=1in]{geometry}
\usepackage{setspace,fullpage}
\usepackage{graphicx}
\usepackage[nice]{nicefrac}
\usepackage{amssymb}
\usepackage{epstopdf}
\usepackage{amsmath,amsthm,amscd,amssymb, mathrsfs}

\usepackage{latexsym}
\usepackage[colorlinks,citecolor=red,pagebackref,hypertexnames=false]{hyperref}

\numberwithin{equation}{section}

\theoremstyle{plain}
\newtheorem{theorem}{Theorem}[section]
\newtheorem{lemma}[theorem]{Lemma}
\newtheorem{corollary}[theorem]{Corollary}

\theoremstyle{definition}

\theoremstyle{remark}

\newtheorem{case[theorem]}{Case}

\title[\parbox{14cm}{\centering{Pinned distance sets, $k$-simplices, Wolff's exponent in finite fields and sum-product estimates \hspace{1in}}} \quad]{Pinned distance sets, $k$-simplices, Wolff's exponent in finite fields and sum-product estimates}
\author{Jeremy Chapman, M. Burak Erdo\u{g}an, Derrick Hart, Alex Iosevich and Doowon Koh}

\begin{document}
\maketitle

\begin{abstract} An analog of the Falconer distance problem in vector spaces over finite fields asks for the threshold $\alpha>0$ such that $|\Delta(E)| \gtrsim q$ whenever $|E| \gtrsim q^{\alpha}$, where $E \subset {\Bbb F}_q^d$, the $d$-dimensional vector space over a finite field with $q$ elements (not necessarily prime). Here $\Delta(E)=\{{(x_1-y_1)}^2+\dots+{(x_d-y_d)}^2: x,y \in E\}$. The fourth listed author and Misha Rudnev (\cite{IR07}) established the threshold $\frac{d+1}{2}$, and in \cite{HIKR07} Misha Rudnev and the third, fourth and fifth authors of this paper proved that this exponent is sharp in odd dimensions.  In two dimensions we improve the exponent to $\tfrac{4}{3}$, consistent with the corresponding exponent in Euclidean space obtained by Wolff (\cite{W99}). 

The pinned distance set $\Delta_y(E)=\{{(x_1-y_1)}^2+\dots+{(x_d-y_d)}^2: x\in E\}$ for a pin $y\in E$ has been studied in the Euclidean setting.  Peres and Schlag (\cite{PS00}) showed that if the Hausdorff dimension of a set $E$ is greater than $\tfrac{d+1}{2}$ then the Lebesgue measure of $\Delta_y(E)$ is positive for almost every pin $y$.  In this paper we obtain the analogous result in the finite field setting.  In addition, the same result is shown to be true for the pinned dot product set $\Pi_y(E)=\{x\cdot y: x\in E\}$. 
Under the additional assumption that the set $E$ has cartesian product structure we improve the pinned threshold for both distances and dot products to $\frac{d^2}{2d-1}$.   

The pinned dot product result for cartesian products implies the following sum-product result.
Let $A\subset \mathbb F_q$ and $z\in \mathbb F^*_q$.  If $|A|\geq q^{\frac{d}{2d-1}}$ then there exists a subset $E'\subset A\times \dots \times A=A^{d-1}$ with $|E'|\gtrsim |A|^{d-1}$ such that  for any $(a_1,\dots, a_{d-1}) \in E'$,
$$|a_1A+a_2A+\dots +a_{d-1}A+zA| > \frac{q}{2} $$ where $a_j A=\{a_ja:a \in A\},j=1,\dots,d-1$.

A generalization of the Falconer distance problem is determine the minimal  $\alpha>0$ such that $E$ contains a congruent copy of every $k$ dimensional simplex whenever $|E| \gtrsim q^{\alpha}$. Here the authors improve on known results (for $k>3$) using Fourier analytic methods, showing that $\alpha$ may be taken to be $\frac{d+k}{2}$.
\end{abstract}         
\vspace{-.4in}
\tableofcontents
\setstretch{1.25}

\section{Introduction} 

The classical Erd\H os distance problem asks for the minimal number of distinct distances determined by a finite point set in ${\Bbb R}^d$, $d \ge 2$. The continuous analog of this problem, called the Falconer distance problem asks for the optimal threshold such that the set of distances determined by a subset of ${\Bbb R}^d$, $d \ge 2$, of larger dimension has positive Lebesgue measure. It is conjectured that a set of $N$ points in ${\Bbb R}^d$, $d \ge 2$, determines $ \gtrapprox N^{\frac{2}{d}}$ distances and, similarly, that a subset of ${\Bbb R}^d$, $d \ge 2$, of Hausdorff dimension greater than $\tfrac{d}{2}$ determines a set of distances of positive Lebesgue measure. Here, and throughout, $X \lessapprox Y$ means that for every $\epsilon>0$ there exists $C_{\epsilon}>0$ such that $X \leq C_{\epsilon}N^{\epsilon}Y$. Similarly, $X \lesssim Y$ means that there exists $C>0$ such that $X \leq CY$. 

Neither problem is close to being completely solved. See \cite{KT04} and \cite{SV05}, and the references contained therein, on the latest developments on the Erd\H os distance problem. See \cite{Erd05} and the references contained therein for the best known exponents for the Falconer distance problem. 

In vector spaces over finite fields, one may define for $E \subset {\Bbb F}_q^d$, 
$$ \Delta(E)=\{||x-y||: x,y \in E\},$$ where 
$$ \|x-y\|={(x_1-y_1)}^2+\dots+{(x_d-y_d)}^2,$$ and one may again ask for the smallest possible size of $\Delta(E)$ in terms of the size of $E$. While $||\cdot||$ is not a distance, in the sense of metric spaces, it is still a rigid invariant in the sense that if $||x-y||=||x'-y'||$, there exists $\tau \in {\Bbb F}_q^d$ and $O \in SO_d({\Bbb F}_q)$, the group of special orthogonal matrices, such that $x'=Ox+\tau$ and $y'=Oy+\tau$. 

There are several issues to contend with here. First, $E$ may be the whole vector space, which would result in the rather small size for the distance set: 
$$ |\Delta(E)|={|E|}^{\frac{1}{d}}.$$ 

Another compelling consideration is that if $q$ is a prime congruent to $1~ (mod~4)$, then there exists $i \in {\Bbb F}_q$ such that $i^2=-1$. This allows us to construct a set in ${\mathbb F}_q^2,$
$$ Z=\{(t,it): t \in {\Bbb F}_q\}$$ and one can readily check that 
$$ \Delta(Z)=\{0\}.$$  

The first  non-trivial result on the Erd\H os-Falconer distance problem in vector spaces over finite fields is proved by Bourgain, Katz and Tao in \cite{BKT04}. The authors get around the first mentioned obstruction by assuming that $|E| \lesssim q^{2-\epsilon}$ for some $\epsilon>0$. They get around the second mentioned obstruction by mandating that $q$ is a prime $\equiv 3~ (mod~4)$. As a result they prove that 
$$ |\Delta(E)| \gtrsim {|E|}^{\frac{1}{2}+\delta},$$ where $\delta$ is a function of $\epsilon$. 

In \cite{IR07} the fourth author along with M. Rudnev went after a distance set result for general fields in arbitrary dimension with explicit exponents. In order to deal with the obstructions outlined above, they reformulated the question in analogy with the Falconer distance problem: how large does $E \subset {\Bbb F}_q^d$, $d \ge 2$, need to be to ensure that $\Delta(E)$ contains a positive proportion of the elements of ${\Bbb F}_q$. They proved that if $|E| \ge 2q^{\frac{d+1}{2}}$, then $\Delta(E)={\Bbb F}_q$ directly in line with Falconer's result (\cite{Fa86}) in Euclidean setting  that for a set $E$ with Hausdorff dimension greater than $\frac{d+1}{2}$ the distance set is of positive measure.   At first, it seemed reasonable that the exponent $\tfrac{d+1}{2}$ may be improvable, in line with the Falconer distance conjecture described above. However, the third,  fourth, and fifth authors of this paper along with M. Rudnev discovered in \cite{HIKR07} that the arithmetic of
  the problem makes the exponent $\tfrac{d+1}{2}$ best possible in {\bf odd dimensions}, at least in general fields. In even dimensions it is still possible that the correct exponent is $\tfrac{d}{2}$, in analogy with the Euclidean case.  In this paper the authors take a first step in this direction by showing that if $|E| \subset {\Bbb F}_q^2$ satisfies $|E| \ge  q^{\frac{4}{3}}$ then $|\Delta(E)| \ge cq$. This is in line with Wolff's result for the Falconer conjecture in the plane which says that the Lebesgue measure of the set of distances determined by a subset of the plane of Hausdorff dimension greater than $\frac{4}{3}$ is positive. 

In \cite{PS00} Peres and Schlag studied the ``pinned" distance sets $\Delta_y(E)=\{\|x-y\|: x \in E\}$ for a ``pin" $y \in E$.  They showed that Falconer's result (\cite{Fa86}) could be sharpened to show that if the Hausdorff dimension of a set $E$ is greater than $\tfrac{d+1}{2}$ then the Lebesgue measure of $\Delta_y(E)$ is positive for almost every pin $y \in E$.  In this paper the authors obtain the analogous result in the finite field setting. In addition, the authors show that this result holds for the pinned dot product sets  $\Pi_y(E)=\{x \cdot y: x \in E\}$

The example which shows that the $\frac{d+1}{2}$ is sharp in odd dimensions is very radial in nature and this led the authors of this paper to consider classes of sets that possess a certain amount of product structure.  In $d$ dimensions we show that for a positive proportion of pins one may obtain a positive proportion of pinned distances for product sets, with the exponent ${\frac{d^2}{2d-1}}$ in place of ${\frac{d+1}{2}}$, improving an analog the exponent due to the second listed author  (\cite{Erd05}) in Euclidean space which holds for all sets.   In the case of pinned dot product sets of subsets with product structure the same result is shown to hold.  This result gives as a direct corollary a result which lies squarely inside a class of problems known as the sum-product problems.  These problems deal with showing in the context of a ring that in a variety of senses multiplicative structure is incompatible with additive structure.

A classical result due to Furstenberg, Katznelson and Weiss (\cite{FKW90}) states that if $E\subset \mathbb R^2$ 
positive upper Lebesgue density, then for any $\delta> 0$, the $\delta$-neighborhood of $E$ contains a 
congruent copy of a sufficiently large dilate of every three-point configuration. 
For arbitrary three-point configurations it is not possible to replace the thickened set $E_{\delta}$ by $E$.  This is due to Bourgain (\cite{B86}) who 
gave an example of a degenerate triangle where all three vertices are on the same line whose large dilates could not be placed in $E$. 
 In the case of $k$-simplex, that is the $k+1$ points spanning a $k$-dimensional subspace, Bourgain (\cite{B86}) applied Fourier analytic techniques to prove that a set $E$ of positive upper Lebesgue density will always contain a sufficiently large dilate of every non-degenerate $k$-point configuration where $k<d$.  If $k\geq d$, it is not currently known whether the $\delta$-neighborhood assumption is necessary.
 
In the case of the integer lattice $\mathbb Z^d$ this problem has been explored as well. Using  Fourier analytic methods \'Akos Magyar proved (\cite{M06}, \cite{M07}) that a set of positive density will contain an congruent copy of every large dilate of a non-degenerate $k$-simplex where $d >2k+4$.

In combinatorics and geometric measure theory the study of $k$-simplices up to congruence may be rephrased in terms of distances. By elementary linear algebra, asking whether a particular translated and rotated copy of a $k$-simplex occurs in a set $E$ is equivalent to asking whether the set of ${k+1 \choose 2}$ distances determined by that $k+1$-point configuration is also determined by some $k+1$ point subset of $E$. In the case of a $1$-simplex this is equivalent to the already discussed Erd\H os and Falconer distance problems.

In the case of vector spaces over finite fields one may then phrase the following generalization of the Erd\H os-Falconer distance problem. How large does $E$ need to be to ensure that $E$ contains a congruent copy of every or at least a positive proportion of all $k$-simplices?
Observe that dilations are not used because the lack of order in in a finite field makes the notion of a sufficiently large dilation meaningless.  

The first investigation into this was done by the third and forth listed authors in \cite{HI07} (see also \cite{HIKSU08}).  It was shown that if a subset
$E$ of $\mathbb F_q^d, d>{k+1 \choose 2}$ is of such that $|E|\gtrsim q^{\frac{k}{k+1} d+\frac{k}{2}}$ then $E$ contains a congruent copy of every $k$-simplices (as long as one is willing to ignore simplices with zero distances). 
This was improved using graph theoretic methods by L. A. Vinh (\cite{V08I}) who obtained the same conclusion for $E$ such that $|E|\gtrsim q^{\frac{d-1}{2}+k}, d\geq 2k$. When the number of points is very close to $d$ these results are trivial.  In the case of triangles in $\mathbb F_q^2$ the third and forth listed authors along with D. Covert and I. Uriarte-Tuero (\cite{CHIU08}) showed that if E has density greater than $\rho$ for some $C q^{-1/2}\leq \rho \leq 1$ with a sufficiently large constant $C > 0$, then the set of triangles determined by $E$, up to congruence, has density greater than $c \rho$.  L. A. Vinh (\cite{V08II}) has shown that for $|E| \gtrsim q^{\frac{d+2}{2}}$ then the set of triangles, up to congruence, has density greater than $c$.  

In this paper the authors show that  for 
$|E|\gtrsim q^{\frac{d+k}{2}}, d\geq k$ then the set of $k$-simplices, up to congruence, has density greater than $c$. We obtain a stronger result in the following situation. Suppose that $E$ is a subset of the $d$-dimensional sphere $S$ where $S=\{x\in\mathbb F_q^d:\|x\|=1\}$.  We show that if $|E|\gtrsim q^{\frac{d+k-1}{2}}$ then $E$ contains a congruent copy of a positive proportion of all $k$-simplices.

The only meaningful sharpness example we have at this point is the Cartesian product of sub-spaces. If $q=p^2$, then there exists a subset of ${\Bbb F}_q^d$ of size exactly $q^{\frac{d}{2}}$ such that all the distances among the vertices of a $k$-simplex are elements of ${\Bbb F}_p$ and thus a positive proportion of $k$-simplexes cannot possibly be realized. On the other hand, in ${\Bbb R}^d$, a conjecture due to Erd\H os and Purdy (see \cite{AS02} and \cite{AAPS07} and the references contained therein) says that an $n$ point set contains fewer than $O(n^{\frac{d}{2}})$ copies of a a $k$-simplex. The classical Lenz construction shows that this estimate would be best possible. It follows that a $n$-point set determines at least $Cn^{k+1-\frac{d}{2}}$ non-congruent $k$-simplexes. The most ambitious conjecture one might be tempted to formulate based on these observations in ${\Bbb F}_q^d$ is that $E \subset {\Bbb F}_q^d$ determines a positive proportion of all the $k$-simplexes, up to congruence, if 
$$ |E| \gtrsim \max \left\{q^{\frac{d}{2}}, q^{\frac{k+1}{k+1-\frac{d}{2}}} \right\}.$$ 

Unfortunately, as we pointed out above, this already fails in the case $k=1$ where the exponent $\frac{d+1}{2}$ is best possible in odd dimensions. We conjecture that in odd dimensions, the exponent $\frac{d+k}{2}$, obtained in this paper, is sharp. In even dimensions, we believe the exponent $\frac{d+k-1}{2}$ to be best possible. 

\section{Statement of Results}

\subsection{Wolff's exponent in finite fields}

Define 
$$\operatorname M_E(q)=\frac{q^{3d+1}}{{|E|}^4} \sum_{t \in {\Bbb F}_q^{*}} \sigma^2_E(t),$$ where 
$$ \sigma_E(t)=\sum_{||m||=t} {|\widehat{E}(m)|}^2.$$ 

In \cite{IR07} the following result is given that gives us a lower bound on the size of the distance set in terms of the upper bound on $\operatorname M_E(q)$. 
\begin{theorem} \label{mattila} Let $E \subset {\mathbb F}_q^d$, $d \ge 2$. Suppose that $|E| \ge C q^{\frac{d}{2}}$ with $C$ sufficiently large. Then
$$ |\Delta(E)| \ge c \min \left\{q, \frac{q}{\operatorname M_E(q)} \right\}.$$
\end{theorem}
In this paper the authors show that in the case of two dimensions one may give a slightly more explicit version of Theorem \ref{mattila}.  An upper bound  on $\operatorname M_E(q)$ of $\sqrt{3}|E|^{-\frac{3}{2}}q^2$ is obtained, which yields that if $E \subset {\Bbb F}_q^2$ with $|E| \geq q^{\frac{4}{3}}$, then 
$ |\Delta(E)| \ge cq.$ In more detail, we have the following result.
\begin{theorem} \label{wolffin2d}  Let $E\subset {\mathbb F}_q^2.$ If $q\equiv 3~(mod~4)$ and $|E|\geq q^{4/3}$, then
$$ |\Delta(E)|> \frac{q}{1+\sqrt{3}}.$$
On the other hand, given $q\equiv 1~(mod~4)$ sufficiently large and $|E|\geq q^{4/3}$, there exists $0< \varepsilon_{q}<1$ such that
$$ |\Delta(E)|> \varepsilon_{q} \cdot q,$$
where $ \varepsilon_{q} \to \frac{1}{1+\sqrt{3}}$ as $q \to \infty.$ In fact, we can choose a $\varepsilon_q$ as the following:
$$\varepsilon_q= \frac{\left(1-2q^{-1}\right)^2}{1+\sqrt{3}-\sqrt{3}q^{-2/3}}.$$
\end{theorem}

\subsection{Pinned distances and dot products} 
Given $y \in {\Bbb F}_q^d$, define the pinned distance set by  
$$\Delta_y(E)=\{||x-y||: x \in E\}.$$ We have the following result.
\begin{theorem} \label{distpinned} Let $E \subset {\Bbb F}_q^d , d\geq 2$. Suppose that 
$ |E|\ge q^{\frac{d+1}{2}}$. Then there exists a subset $E'$ of $E$ with $|E'|\gtrsim |E|$ such that for every $y \in E'$ one has that 
$$ |\Delta_y(E)|>\frac{q}{2}.$$
\end{theorem} 

In analogy with the pinned distance set define the pinned dot product set by  
$$\Pi_y(E)=\{x \cdot y: x \in E\}.$$ 
\begin{theorem} \label{dotpinned} Let $E \subset {\Bbb F}_q^d$. Suppose that 
$ |E|\ge q^{\frac{d+1}{2}}$. Then there exists a subset $E'$ of $E$ with $|E'|\gtrsim |E|$ such that for every $y \in E'$ one has that $|\Pi_y(E)| > \frac{q}{2}$.  
\end{theorem}

\subsection{Cartesian Products}\label{cartprod}

Let $\pi(x)=(x_1, \dots, x_{d-1})$ and define 
$$ E_z=\pi(E) \times \{z\},$$ where $z$ is an element of $\mathbb F_q$ and $x\in {\mathbb F}_q^d.$ Here we could have chosen to place $z$ in any coordinate and have chosen to put $z$ in the $d$th coordinate only for simplicity of notation.

Given $y\in \pi(E) \times {\mathbb F}_q$ and $z \in {\mathbb F}_q,$ define 
$$ \Delta^{(z)}_y(E)=\{||x-\tilde{y}||:x \in E \},$$
where $ \tilde{y}= (\pi(y),z)\in E_z.$ 

We have the following result. 
\begin{theorem} \label{distpinnedproj} Let $E \subset {\Bbb F}_q^d$ and let $E_z $ be defined with respect to the projection $\pi$, 
and an element $z \in \mathbb F_q$ as above. Suppose that 
$$ |E||E_z| \ge q^d.$$ 
Then there exists a $E'_z \subset E_z$ with $|E'_z|\gtrsim |E_z|$ such that for every $ (\pi (y), z )\in E'_z , $

\begin{equation*}  |\Delta^{(z)}_y(E)| > \frac{q}{3}.\end{equation*}
\end{theorem} 
Given $E\subset {\mathbb F}_q^d,$ we define 
$$P(E)=\{z\in {\mathbb F}_q: (\pi(y),z)\in E~~\mbox{for some}~y\in {\mathbb F}_q^d\}.$$ The set $P(E)$ is composed of all the last coordinates of elements in $E.$  
Observe that if $E$ is a product set, then $E_z\subset E$ for all $z\in P(E)$ and $\bigcup_{z\in P(E)}E_z=E.$ Moreover, $E_z$ and $E_{z^\prime}$ are disjoint if $ z\neq z^\prime$ and $|E_z|=|E_{z^\prime}|.$
This leads us to the following consequence of Theorem \ref{distpinnedproj}. 
\begin{corollary} \label{distpinnedprod} Suppose that $E=A_1 \times A_2 \times \dots \times A_d$, where $A_j$ is contained in ${\Bbb F}_q$. Suppose that 
$$ |E| \ge q^{\frac{d^2}{2d-1}}.$$ 
Then there exists a subset $E'\subset E$ with $|E'|\gtrsim |E|$ such that for every $y\in E'$ one has that
$$ |\Delta_y(E)| > \frac{q}{3}.$$ 
\end{corollary} 

The Corollary immediately follows  from Theorem \ref{distpinnedproj}. To see this, since $E$ is a product set,
 after perhaps relabeling some coordinates, we may assume, using straightforward pigeon-holing, that $E=\pi(E) \times P(E)$, 
where $|\pi(E)| \ge {|E|}^{\frac{d-1}{d}},$ and we have $|\pi(E)|=|E_z|$ for all $z\in P(E).$ 
Since $|E|\ge q^{\frac{d^2}{2d-1}}$, we see that $|E||E_z|\geq q^d$ for all $z\in P(E).$
Applying Theorem \ref{distpinnedproj}, we can choose the set $E^\prime_z$ for all $z\in P(E)$ which satisfies the conclusion 
of Theorem \ref{distpinnedproj}. Taking $ E^\prime= \bigcup_{z\in P(E)} E^\prime_z$ , the proof of Corollary 
\ref{distpinnedprod} is complete.
Observe that we could have made a much weaker, though more technical, assumption on the structure of $E$. 

Given $y\in \pi(E) \times {\mathbb F}_q$ and $z \in {\mathbb F}_q,$
define the pinned dot product set to be 
$$ \Pi^{(z)}_y(E)=\{x\cdot \tilde{y}:x \in E \},$$ 
where $\tilde{y}= (\pi(y), z)\in E_z.$
We use the method of proof of Theorem \ref{distpinnedproj} above to obtain the following. 
\begin{theorem} \label{dotpinnedproj} Let $E \subset {\Bbb F}_q^d$ and let $E_z , z\in {\mathbb F}_q^*,$ be defined as above. Suppose that 
$$ |E||E_z| \ge q^d.$$ 

Then there exists a $E'_z\subset E_z$ with $|E'_z|\gtrsim |E_z|$ such that for every $(\pi(y), z) \in E'_z,$
\begin{equation*}  |\Pi^{(z)}_y(E)| > \frac{q}{2}.\end{equation*}
\end{theorem}

\subsection{Sums and products implications} 

A related line investigation that has received much recent attention is the
following. 
Let $A \subset {\Bbb F}_q$. How large does $A$ need to be to ensure that 
$$ {\Bbb F}_q^{*} \subset \underbrace{A\cdot A+\ldots+A\cdot A}_{d\ \mbox{{\small times} }}$$ or, more modestly, 
$$ |A \cdot A+\ldots+A \cdot A| \ge cq$$ for some $c>0$. 

A result due to Bourgain (\cite{Bo05}) gave the following answer to this question.
\begin{theorem}  Let $A$ be a subset of $\mathbb F_q$ such that $|A|\geq C q^{\frac{3}{4}}$ then $A\cdot A+A\cdot A
+A\cdot A =\mathbb F_q$.
\end{theorem}

  Due to the misbehavior of the zero element it is not possible for $A\cdot A+A\cdot A=\mathbb F_q$ unless $A$ is a positive proportion of the elements of $\mathbb F_q$. However, it is
reasonable to conjecture that if $|A| \ge
C_{\epsilon}q^{\frac{1}{2}+\epsilon}$, then $A\cdot A+A\cdot A\supseteq \mathbb F^*_q$. This
result cannot hold, especially in the setting of general finite fields
if $|A|=\sqrt{q}$ because $A$ may in fact be a subfield. See also
\cite{BGK06}, \cite{Cr04},  \cite{TV06} and the references contained
therein on recent progress related to this problem and its analogs.
For example, Glibichuk and Konyagin, \cite{GK06} (see also \cite{Gl06}), proved in the case of prime fields ${\mathbb Z}_p$ that for $|A|>\sqrt{p}$
that on case take $d=8$.  This was extended to arbitrary finite fields by
Glibichuk in \cite{GK08}.
These results were
achieved by methods of arithmetic combinatorics.

The third and fourth listed authors used character sum machinery to obtain the following result.
\begin{theorem} \label{kickass} Let $A \subset {\mathbb F}_q^*$. \begin{itemize} \item
If $|A|>q^{\frac{1}{2}+\frac{1}{2d}}$ then $A \cdot A+\ldots+A \cdot A \supseteq {\mathbb F}_q^{*} .$
\item If $|A|\geq q^{\frac{1}{2}+\frac{1}{2(2d-1)}}$ then $|A \cdot A+\ldots+A \cdot A|\geq \frac{1}{2}q .$ 
\end{itemize} 
\end{theorem}

In view of Glibichuk' result (\cite{GK08}) one may note that Theorem \ref{kickass} is only interesting in the case $d<8$.  It follows immediately that in the perhaps the most
interesting case $d=2$, that $A\cdot A+A\cdot A \supseteq {\mathbb F}_q^{*} $ for $|A|>q^{\frac{3}{4}},$ and
$ |A\cdot A+A\cdot A| \geq \frac{q}{2}$ for $|A| > q^{\frac{2}{3}}.$  One may note that if $A\cdot A+A\cdot A=\mathbb F_q^*$ then
 only a minimal amount of additional additive structure is needed to get the zero element.  Specifically, if $|A|>q^{\frac{3}{4}}$ and $B$ is any subset of $\mathbb F_q$ with $|B|>1$ then $A\cdot A+A\cdot A+B=\mathbb F_q$.

Shparlinski (\cite{S07}) using multiplicative character sums showed that 
 if $ |A| \ge q^{\frac{2}{3}}$ then for any $z \in A$, 
$ |A \cdot A+zA| \ge \frac{q}{2}.$

An immediate implication of Theorem \ref{dotpinnedproj} is the following which says that if a set $A$ is sufficiently robust then a large class of linear equations have solutions in $A$.
\begin{theorem}
Let $A\subset \mathbb F_q$ and $z\in \mathbb F^*_q$. If $|A|\geq q^{\frac{d}{2d-1}},$ then
 there exists a subset $E'\subset A\times \dots \times A=A^{d-1}$ with $|E'|\gtrsim |A|^{d-1}$ such that  for any $(a_1,\dots, a_{d-1}) \in E'$,
$$|a_1A+a_2A+\dots +a_{d-1}A+zA| > \frac{q}{2} $$
 where $a_j A=\{a_ja:a \in A\},j=1,\dots,d-1$.
\end{theorem}

\subsection{$k$-simplices}
Let $P_k$ denote a $k$-simplex, that is $k+1$ points spanning a $k$ dimensional subspace.
Given another $k$-simplex $P'_k$ we write $P'_k\sim P_k$ if  there exists a $\tau \in {\Bbb F}_q^d$ and an
$O \in SO_d({\Bbb F}_q)$, the set of $d$-by-$d$ orthogonal matrices over ${\Bbb F}_q$ such that 
$$ P'_k=O (P_k)+\tau.$$ 
For $E \subset {\Bbb F}_q^d$ define
$$\mathcal{T}_k(E)=\{P_k \in E \times \dots \times E\} \ / \sim.$$ 
 
Under this equivalence relation one may specify a simplex by the distances determined by its vertices.  This follows from the following simple lemma from \cite{HI07}. 
\begin{lemma}\label{uptocrap} Let $P_k$ be a simplex with vertices
$V_0, V_1, \dots, V_k$, $V_j \in {\Bbb F}_q^d$. Let $P'$ be another
simplex with vertices $V'_0, V'_1, \dots, V'_k$. Suppose that
\begin{equation} \label{equalnorm} ||V_i-V_j||=||V'_i-V'_j|| \end{equation}
for all $i,j$.
Then there exists $\tau \in {\Bbb F}_q^d$ and $O \in SO_d({\Bbb F}_q)$ such
that $\tau+O(P)=P'$.
\end{lemma}
 
 In this paper the authors will specify simplices by specifying the distances determining them piece by piece.  With this in mind denote a $k$-star by
 $$S_k(t_1,\dots,t_k)=\{(x,y^1\dots y^k): \|x-y^1\|=t_1,\dots \|x-y^k\|=t_k\},$$
 where $t_1,\dots, t_k \in \mathbb F_q$.
 
 Define
$\Delta_{y^1, y^2,\dots,y^{k}}(E)=\{(\|x-y^1\|, \dots,\|x-
y^{k}\|)\in {\mathbb F}_q^k:x \in E\}$
where $y^1$, $y^2$,\dots,$y^k \in E$.  
We have the following result. 
   \begin{theorem} \label{distmultproj}   Let $E \subset {\Bbb F}_q^d$. 
  If $|E| \gtrsim q^{\frac{d+k}{2}}$ then
$$\frac{1}{|E|^{k}}\sum_{y^1,\dots,y^{k}\in E}|\Delta_{y^1,\dots,y^k}(E)|\gtrsim q^{k}.$$
\end {theorem}

An  pigeon-holing argument using Theorem \ref{distmultproj} will allow us to move from sets of $k$-stars to sets of $k$-simplices.

\begin{theorem}\label{kpoint}
Let $E\subset \mathbb F_q^d$.  If $|E|\gtrsim q^{\frac{d+k}{2}},k\leq d$ then $|\mathcal{T}_k(E)|\gtrsim q^{k+1 \choose 2}$, in other words $E$ determines a positive proportion of all $k$-simplices.
\end{theorem}

Similarly, define
$\Pi_{y^1, y^2,\dots,y^{k}}(E)=\{(x \cdot y^1, x \cdot y^2,\dots,x \cdot y^{k})\in {\mathbb F}_q^{k}:x \in E\}$
where $y^1$, $y^2$,\dots,$y^{k} \in E$.  
 Then we have the following result.
 \begin{theorem} \label{dotmultproj}   Let $E \subset {\Bbb F}_q^d$. 
  If $|E| \gtrsim q^{\frac{d+k}{2}}$ then
$$\frac{1}{|E|^{k}}\sum_{y^1,\dots,y^{k}\in E}|\Pi_{y^1,\dots,y^{k}}(E)|\gtrsim q^{k}.$$
\end {theorem}

If $E$ is subset of a sphere $S$ where $S=\{x\in \mathbb F_q^d:\|x\|=1\}$ then one has for $x,y \in E$ that $\|x-y\|=2-2x\cdot y$.  Therefore in this case determining distances is the same as determining dot products.
Under this assumption on $E$ the proof of Theorem \ref{dotmultproj} may be modified improving the exponent in Theorem \ref{distmultproj}.
 \begin{theorem} \label{distmultprojsphere}   Let $E \subset S$. 
  If $|E| \gtrsim q^{\frac{d+k-1}{2}}$ then
$$\frac{1}{|E|^{k}}\sum_{y^1,\dots,y^{k}\in E}|\Delta_{y^1,\dots,y^k}(E)|\gtrsim q^{k}.$$
\end {theorem}
This in turn yields the following result.
\begin{theorem}\label{kpointsphere}
Let $E\subset S$.  If $|E|\gtrsim q^{\frac{d+k-1}{2}}~,k\leq d-1$ then $|\mathcal{T}_k(E)|\gtrsim q^{k+1 \choose 2}$, in other words $E$ determines a positive proportion of all $k$-simplices.
\end{theorem}
The proof of this theorem we will omit follows directly that of Theorem \ref{kpoint}.

\section{Finite field Fourier transform}

Recall that given a function $f: {\Bbb F}_q^d \to {\Bbb C}$, the Fourier transform with respect to a non-trivial additive character $\chi$ on ${\Bbb F}_q$ is given by the relation 
\begin{equation*} \label{ftdef} \widehat{f}(m)=q^{-d} \sum_{x \in {\Bbb F}_q^d} \chi(-x \cdot m) f(x). \end{equation*} Also recall that the Fourier inversion theorem is given by
\begin{equation*} \label{inversion} f(x)=\sum_{m \in {\Bbb F}_q^d} \chi(x \cdot m) 
\widehat{f}(m) \end{equation*} and the Plancherel theorem is given by
\begin{equation*} \label{plancherel} \sum_{m \in {\Bbb F}_q^d} {|\widehat{f}(m)|}^2=q^{-d} \sum_{x \in {\Bbb F}_q^d} {|f(x)|}^2 \end{equation*} 

For a subset $E$ of $\mathbb F_q^d$ we will use $E(x)$ to denote the indicator function of $E$.

\section{Proof of Theorem \ref{wolffin2d} - Wolff's exponent}
This section contains two subsections.
In the first subsection we obtain main lemmas for the proof of Theorem \ref{wolffin2d}.
The complete proof of Theorem \ref{wolffin2d} is given in the second subsection.

\subsection{Lemmas for the proof of Theorem \ref{wolffin2d}}
We begin by defining the counting function,
$$\nu(t)=\sum_{\|x-y\|=t} E(x)E(y).$$
Then we may write
$$\nu(t)=\sum_{x ,y\in E} S_t(x-y),$$
where $S_t$ is the sphere of radius $t$, $\{x\in \mathbb F_q^d:\|x\|=t\}$.

We first obtain some information about $\nu(t)$.

\begin{lemma}\label{main} Let $E\subset {\mathbb F}_q^2.$ Then we have
$$\sum_{t\in {\mathbb F}_q} \nu^2(t)= q^6\sum_{t\in {\mathbb F}_q} \left(\sum_{m\in S_t} |\widehat{E}(m)|^2 \right)^2 +q^{-1}|E|^4-q|E|^2.$$
\end{lemma}
\begin{proof}
Using the Fourier inversion theorem of $S_t(x-y)$ and definition of the Fourier transform, we have
\begin{equation}\label{zerodistance}\nu(t)=\sum_{x,y\in E} S_t(x-y)= q^4 \sum_{m\in {\mathbb F}_q^2} |\widehat{E}(m)|^2 \widehat{S_t}(m).\end{equation}

It follows that
$$ \nu^2(t)=q^8 \sum_{m,m'\in {\mathbb F}_q^2} |\widehat{E}(m)|^2 |\widehat{E}(m')|^2 \widehat{S_t}(m) \overline{\widehat{S_t}}(m')$$
$$=q^8 |\widehat{E}(0,0)|^4 |\widehat{S_t}(0,0)|^2 + 2 q^8 \sum_{m\in {\mathbb F}_q^2\setminus (0,0)} |\widehat{E}(m)|^2\widehat{S_t}(m) |\widehat{E}(0,0)|^2 \overline{\widehat{S_t}}(0,0)$$
$$+ q^8 \sum_{m,m'\in {\mathbb F}_q^2\setminus (0,0)} |\widehat{E}(m)|^2 |\widehat{E}(m')|^2 \widehat{S_t}(m) \overline{\widehat{S_t}}(m')= I(t)+II(t)+III(t).$$
Since $ \widehat{E}(0,0)=q^{-2}|E|$ and $ \widehat{S_t}(0,0)= q^{-2}|S_t|,$ we obtain
\begin{equation}\label{formulaI}\sum_{t\in{\mathbb F}_q} I(t)=q^{-4}|E|^4\sum_{t\in {\mathbb F}_q} |S_t|^2.\end{equation}
We will need the following lemma which we will delay proving until the last section.

\begin{lemma}\label{sphere} Let $S_t\subset {\mathbb F}_q^d.$ Then we have
$$\sum_{t\in {\mathbb F}_q} |S_t|^2= q^{2d-1} +q^{d}-q^{d-1},$$
and also for $m \in {\mathbb F}_q^d \setminus \{0, \ldots, 0)$,
 $$\sum_{t\in {\mathbb F}_q} |\widehat{S_t}(m)|^2=  q^{-d}-q^{-d-1},$$
 and
 $$\sum_{t\in {\mathbb F}_q} |S_t| \widehat{S}_t(m)\leq  1-  q^{-1}.$$
\end{lemma}

The first part of Lemma \ref{sphere} together with (\ref{formulaI}) yields the following equality:
\begin{equation}\label{I}
\sum_{t\in {\mathbb F}_q} I(t)=|E|^4(q^{-1}+q^{-2}-q^{-3}).\end{equation}

Now we compute the $\sum_{t\in {\mathbb F}_q} II(t).$ It follows that
\begin{equation}\label{formulaII} \sum_{t\in {\mathbb F}_q} II(t)= 2q^2 |E|^2 \sum_{m\neq (0,0)} |\widehat{E}(m)|^2\sum_{t\in {\mathbb F}_q} |S_t|\widehat{S_t}(m).\end{equation}
We claim that if the dimension $d$ is even, $S_t \subset {\mathbb F}_q^d,$ and $ m\in {\mathbb F}_q^d \setminus (0,\ldots,0)$, then we have
$$ \sum_{t\in {\mathbb F}_q} |S_t| \widehat{S_t}(m)=q^{(-d-2)/2} \psi\left((-1)^{d/2}\right) G_1^d(\psi,\chi) \sum_{s\neq 0} \chi\left(\frac{\|m\|}{4s}\right),$$
where $\psi$ is the quadratic character of order two and $ G_1(\psi, \chi)$ is the Gauss sum given by $ G_1(\psi, \chi)= \sum_{s\neq 0} \psi(s)\chi(s).$ The claim follows from the proof of the third part of Lemma \ref{sphere} (see the proof of Lemma \ref{sphere} in the last section).
We also need the following theorem .
\begin{theorem}[Theorem 5.15 in \cite{LN97}]\label{ExplicitGauss}
Let ${\mathbb F}_q$ be a finite field with $ q= {p}^l$, where $p$ is an odd prime and $l \in {\mathbb N}.$
Let $\psi$ be the quadratic character of ${\mathbb F}_q$ and let $\chi$ be the canonical additive character of ${\mathbb F}_q$.
Then we have
$$G_1(\eta, \chi)= \left\{\begin{array}{ll}  {(-1)}^{l-1} q^{\frac{1}{2}} \quad &\mbox{if} \quad p =1 \,\,(mod~4) \\
                    {(-1)}^{l-1} i^l q^{\frac{1}{2}} \quad &\mbox{if} \quad p =3 \,\,(mod~4).\end{array}\right. $$ \end{theorem}
Using Theorem \ref{ExplicitGauss}, we see that if $d$ is even, then $\psi\left((-1)^{d/2}\right) G_1^d(\psi,\chi) =q^{d/2},$ because
$\psi(-1)=1$ if $ q\equiv 1~(mod~4)$ and $\psi(-1)=-1$ if $q=3~(mod~4).$
Thus if $d=2$ and $m\neq (0,0)$, we have
$$ \sum_{t\in {\mathbb F}_q} |S_t|\widehat{S_t}(m)= q^{-1}\sum_{s\neq 0} \chi\left(\frac{\|m\|}{4s}\right).$$
Plugging this into (\ref{formulaII}), we have
\begin{align*} \sum_{t\in {\mathbb F}_q} II(t)&= 2q|E|^2 \sum_{m\neq (0,0)} |\widehat{E}(m)|^2\sum_{s\neq 0} \chi\left(\frac{\|m\|}{4s}\right)\\
&=2q|E|^2 \left( \sum_{m\neq(0,0): \|m\|= 0}|\widehat{E}(m)|^2(q-1)+\sum_{m\neq (0,0):\|m\|\neq 0}(-1) |\widehat{E}(m)|^2 \right)\\
&=2q|E|^2 \left( q \sum_{m\neq(0,0): \|m\|=0}|\widehat{E}(m)|^2- \sum_{m\neq (0,0)} |\widehat{E}(m)|^2\right).\end{align*}

Now Replacing  $\sum_{m\neq(0,0): \|m\|=0}|\widehat{E}(m)|^2$ by $\sum_{\|m\|=0} |\widehat{E}(m)|^2- |\widehat{E}(0,0)|^2$ and
observing by the Plancherel theorem that $\sum_{m\neq (0,0)}|\widehat{E}(m)|^2= \sum_{m\in {\mathbb F}_q^2}|\widehat{E}(m)|^2 - |\widehat{E}(0,0)|^2=q^{-2}|E|-q^{-4}|E|^2$, we obtain
\begin{equation}\label{II} \sum_{t\in {\mathbb F}_q} II(t)=2q^2|E|^2 \sum_{\|m\|=0} |\widehat{E}(m)|^2 -2q^{-1}|E|^3-2q^{-2}|E|^4+2q^{-3}|E|^4.\end{equation}
Finally, we estimate the $ \sum_{t\in {\mathbb F}_q} III(t)$ which is given by
\begin{equation}\label{formulaIII} \sum_{t\in {\mathbb F}_q} III(t)= q^8\sum_{m,m'\neq (0,0)} 
|\widehat{E}(m)|^2|\widehat{E}(m')|^2 \sum_{t\in{\mathbb F}_q} \widehat{S_t}(m)\overline{\widehat{S_t}}(m').\end{equation}
In \cite{IKo08}, the Fourier transform of $S_t$ was given by the formula
\begin{equation}\label{sphereFouriertransform} \widehat{S_t}(m)=q^{-1}\delta_0(m)+ q^{-d-1}\psi^d(-1)G_1^d(\psi,\chi) \sum_{s\neq 0} \chi\left( \frac{\|m\|}{4s} +st\right) \psi^d(s),\end{equation}
where $\delta_0(m)=1$ if $m=(0,\ldots,0)$ and $\delta_0(m)=0$ if $m\neq (0,\ldots,0).$
Using this formula and the orthogonality relation of the non-trivial additive character $\chi$ in $t$-variables, 
we see that for $m,m'\in {\mathbb F}_q^2\setminus (0,0)$,
$$ \sum_{t\in {\mathbb F}_q} \widehat{S_t}(m) \overline{\widehat{S_t}}(m')= q^{-3} \sum_{s\neq 0} \chi\left( \frac{\|m\|-\|m'\|}{4s}\right).$$
Plugging this into (\ref{formulaIII}), we have
$$ \sum_{t\in {\mathbb F}_q} III(t)= q^5 \sum_{m,m'\neq (0,0)}|\widehat{E}(m)|^2|\widehat{E}(m')|^2\sum_{s\neq 0}
 \chi\left(\frac{\|m\|-\|m'\|}{4s}\right).$$
 Using a change of variables, $ 1/(4s)\to s,$ and the properties of the summation notation, we have
\begin{align*} \sum_{t\in {\mathbb F}_q} III(t)&=q^5 \sum_{m,m'\neq (0,0)}|\widehat{E}(m)|^2|\widehat{E}(m')|^2 \left(-1 + \sum_{s\in {\mathbb F}_q} \chi(s(\|m\|-\|m'\|)) \right)\\
&=q^6 \sum_{m,m'\neq (0,0): \|m\|=\|m'\|} |\widehat{E}(m)|^2 |\widehat{E}(m')|^2-q^5 \sum_{m,m'\neq (0,0)} |\widehat{E}(m)|^2 |\widehat{E}(m')|^2\\
&= q^6 \sum_{t\in {\mathbb F}_q} \left( \sum_{m\neq (0,0): \|m\|=t} |\widehat{E}(m)|^2\right)^2 -q^5 \left( \sum_{m\neq (0,0)}|\widehat{E}(m)|^2\right)^2\\
&=q^6 \left(\sum_{m\neq (0,0):\|m\|=0}|\widehat{E}(m)|^2\right)^2 +q^6 \sum_{t\neq 0} \left( \sum_{m\neq (0,0): \|m\|=t} |\widehat{E}(m)|^2\right)^2-q^5\left(q^{-2}|E|-q^{-4}|E|^2\right)^2.\end{align*}

Since $ \sum_{m\neq (0,0):\|m\|=0}|\widehat{E}(m)|^2= \sum_{\|m\|=0} |\widehat{E}(m)|^2 - |\widehat{E}(0,0)|^2,$ and $\widehat{E}(0,0)=q^{-2}|E|,$ a direct calculation yields 
\begin{equation}\label{III} \sum_{t\in {\mathbb F}_q} III(t)=q^6 \sum_{t\in {\mathbb F}_q} \left( \sum_{m\in S_t} |\widehat{E}(m)|^2\right)^2-2q^2|E|^2 \sum_{m\in S_0}|\widehat{E}(m)|^2 +q^{-2}|E|^4+2q^{-1}|E|^3- q|E|^2- q^{-3}|E|^4. \end{equation}

From (\ref{I}),(\ref{II}), and (\ref{III}), the proof of Lemma \ref{main} is complete.
\end{proof}

We now introduce and prove the second key lemma for the proof of Theorem \ref{wolffin2d}.
The following lemma was implicitly given in \cite{IKo07} and we shall follow the outline in \cite{IKo07} to get the following lemma. 

\newpage
\begin{lemma}\label{restriction} If $E$ is a subset of $ {\mathbb F}_q^2,$ then it follows that

$$ \max_{t\in {\mathbb F}_q \setminus \{0\}} \sum_{m\in S_t}|\widehat{E}(m)|^2 \leq  \frac{\sqrt{3} |E|^{3/2}}{q^3} .$$
\end{lemma}
\begin{proof}
The proof is based on the extension theorem related to circles in ${\mathbb F}_q^2.$
In \cite{IKo07}, it was proved that the extension operator for the circle with non-zero radius is bounded from $L^2$ to $L^4$ and the mapping property is sharp. However, the operator norm was not given in the explicit form.
Here, we shall observe the explicit operator norm and derive  Lemma \ref{restriction}. We begin by recalling the meaning of norms and Fourier analysis machinery. We are working in the space $({\mathbb F}_q^2, dx)$ which we endow with the normalized counting measure. Thus if $f$ is defined on the space, then the $L^p$-norm is given by 
$$\|f\|_{L^p({\mathbb F}_q^2,dx)} = \left(q^{-2} \sum_{x\in {\mathbb F}_q^2} |f(x)|^p\right)^{1/p},$$
where $q^2$ is the number of elements of ${\mathbb F}_q^2.$
Recall that the Fourier transform of the function $f$ is actually defined on the dual space of $({\mathbb F}_q^2,dx).$ We denote by $({\mathbb F}_q^2,dm)$ the dual space, which is endowed with the counting measure $dm$. For a non-trivial additive character $\chi$ of ${\mathbb F}_q$, we therefore define the Fourier transform of the function $f$ on $({\mathbb F}_q^2,dx)$ by the formula
$$ \widehat{f}(m)= q^{-2} \sum_{x\in {\mathbb F}_q^2} \chi(-x\cdot m) f(x),$$
where $m$ is considered as an element of the dual space $({\mathbb F}_q^2, dm).$ Taking the different measures between the function space and the dual space, we obtain the Plancherel theorem, that is 
$$ \|\widehat{f}\|_{L^2({\mathbb F}_q^2, dm)}= \|f\|_{L^2({\mathbb F}_q^2, dx)}.$$
Note that this means the following:
$$ \sum_{m\in {\mathbb F}_q^2} |\widehat{f}(m)|^2 = q^{-2} \sum_{x\in {\mathbb F}_q^2} |f(x)|^2,$$
where $f$ is a function on $({\mathbb F}_q^2, dx)$ and $\widehat{f}$ is a function on $({\mathbb F}_q^2, dm).$ We now introduce the normalized curve measure $d\sigma$ on the circle $S_t$ in $({\mathbb F}_q^2, dx).$ The measure $d\sigma$ is defined by the relation
$$ \widehat{fd\sigma}(m)= |S_t|^{-1} \sum_{x\in S_t} \chi(-x\cdot m)f(x),$$
where $f$ is a function on $({\mathbb F}_q^2, dx).$ In fact, the measure $\sigma$ can be considered as the following function on $( {\mathbb F}_q^2,dx)$:
$$ \sigma(x)= q^{2}|S_t|^{-1} S_t(x),$$
where $S_t(x)$ means the characteristic function on $S_t.$

 Using Plancherel, we first observe that
\begin{equation*}\label{number}\|\widehat{fd\sigma}\|^4_{L^4({\mathbb F}_q^2,dm)}= \|fd\sigma \ast fd\sigma\|^2_{L^2({\mathbb F}_q^2,dx)},
\end{equation*}
which is
$$=q^{-2}|fd\sigma \ast fd\sigma (0,0)|^2 + q^{-2}\sum_{x\in {\mathbb F}_q^2\setminus (0,0)} |fd\sigma \ast fd\sigma (x)|^2=I+II.$$

To estimate the term $I$, we note that
$$ |fd\sigma \ast fd\sigma(0,0)|\leq \sum_{m\in {\mathbb F}_q^2} |\widehat{fd\sigma}(m)|^2=q^2|S_t|^{-1} \|f\|^2_{L^2(S_t,d\sigma)}.$$
Thus the term $I$ is estimated by 
\begin{equation}\label{numberI} I\leq q^2 |S_t|^{-2} \|f\|^4_{L^2(S_t, d\sigma)}.\end{equation}
Using H\"older's inequality, we have
\begin{align}\label{numberII}
II= &q^{-2} \sum_{x\in {\mathbb F}_q^2\setminus (0,0)} |fd\sigma \ast fd\sigma(x)|^2\nonumber\\
\leq& \|d\sigma \ast d\sigma\|_{L^\infty \left({\mathbb F}_q^2\setminus (0,0),dx \right)} \|f\|^4_{L^2(S_t,d\sigma)}\nonumber\\
=&\left(\max_{x\neq (0,0)} q^2 |S_t|^{-2} \sum_{(\alpha,\beta)\in S_t\times S_t: \alpha+\beta=x} 1\right)\cdot \|f\|^4_{L^2(S_t,d\sigma)}.
\end{align}

From (\ref{number}), (\ref{numberI}), and (\ref{numberII}), we obtain the following:
\begin{align*}\|\widehat{fd\sigma}\|_{L^4({\mathbb F}_q^2,dm)}\leq &\left( q^2|S_t|^{-2}+ q^2|S_t|^{-2} \max_{x\neq (0,0)}\sum_{(\alpha,\beta)\in S_t\times S_t: \alpha+\beta=x} 1\right)^{1/4} \|f\|_{L^2(S_t,d\sigma)}\\
\leq& \left(3q^2|S_t|^{-2}\right)^{1/4} \|f\|_{L^2(S_t,d\sigma)},\end{align*}
By duality, we have the following restriction estimate: for all complex-valued function $g$ on ${\mathbb F}_q^2,$
\begin{equation}\label{dual} \|\widehat{g}\|_{L^2(S_t, d\sigma)}\leq \left(3q^2|S_t|^{-2}\right)^{1/4} \|g\|_{L^{4/3}({\mathbb F}_q^2, dm)}.\end{equation}
Since the function $g$ above is defined on $({\mathbb F}_q^2, dm)$ with a  counting measure $dm$, the Fourier transform of $g$ is given
$$ \widehat{g}(x)= \sum_{m\in {\mathbb F}_q^2} \chi(-x\cdot m)g(m).$$
Moreover, since $d\sigma$ is a normalized curve measure on the circle $S_t$, we have
$$\|\widehat{g}\|_{L^2(S_t, d\sigma)}=\left( |S_t|^{-1} \sum_{x\in S_t} |\widehat{g}(x)|^2\right)^{1/2}.$$
After taking $g$ as a characteristic function on the set $E\subset ({\mathbb F}_q^2, dm)$ and identifying the space $({\mathbb F}_q^2, dx)$ with the dual space  $({\mathbb F}_q^2, dm)$, the conclusion in Theorem (\ref{restriction}) immediately follows from the inequality (\ref{dual}).
\end{proof}

\subsection{The complete proof of Theorem \ref{wolffin2d}}

We first prove the first part of Theorem \ref{wolffin2d}. 
Applying the Cauchy-Schwarz inequality, we see that
$$ |E|^4=\left(\sum_{t\in {\mathbb F}_q } \nu(t)\right)^2 \leq |\Delta(E)| \sum_{t\in {\mathbb F}_q}\nu^2(t).$$
It follows that 
\begin{equation}\label{formuladist1} |\Delta(E)|\geq \frac{|E|^4}{ \sum_{t\in {\mathbb F}_q}\nu^2(t)}.\end{equation}
Thus our main work is to find the good upper bound of  $\sum_{t\in {\mathbb F}_q}\nu^2(t).$
If $q\equiv 3~(mod~4)$, then the circle $S_0$ with zero radius only contains the origin. From Lemma \ref{main} and Lemma \ref{restriction}, we therefore obtain the following :
\begin{align*} \sum_{t\in {\mathbb F}_q} \nu^2(t)\leq & q^6|\widehat{E}(0,0)|^4 + q^6  \left(\max_{t\neq 0} \sum_{m\in S_t}|\widehat{E}(m)|^2\right) \cdot \sum_{m\neq (0,0)} |\widehat{E}(m)|^2 + q^{-1}|E|^4-q|E|^2\\
\leq & q^6 q^{-8}|E|^4 +q^6 \frac{\sqrt{3} |E|^{3/2}}{q^3} (q^{-2}|E|-q^{-4}|E|^2)+q^{-1}|E|^4-q|E|^2\\
 =& q^{-1}|E|^4+q^{-2}|E|^4-q|E|^2+ \sqrt{3}|E|^{5/2} \left(q-q^{-1}|E|\right). \end{align*}
 If we assume that $q^{4/3}\leq |E|\leq q^{3/2},$ then it is clear that the last term above is less than the value $(1+\sqrt{3}) q^{-1}|E|^4.$ Thus we conclude that for every $ q^{4/3}\leq |E| \leq q^{3/2}$,
\begin{equation}\label{smallset} |\Delta(E)|> \frac{q}{1+\sqrt{3}}.\end{equation}
For $ |E|> q^{3/2}$, the inequality (\ref{smallset}) is clear, because $|\Delta(E')|\leq|\Delta(E)| $ if $E'\subset E.$
Thus we complete the proof of the first part of Theorem \ref{wolffin2d}.
 
 We now prove the second part of Theorem \ref{wolffin2d}.
 We assume that $q\equiv 1~(mod~4).$ Applying the Cauchy-Schwarz inequality, we have 
 \begin{align*}(|E|^2-\nu(0))^2&= \left(\sum_{t\in {\mathbb F}_q\setminus \{0\}} \nu(t)\right)^2\\
 &\leq \left(\sum_{t\neq 0: t\in \Delta(E)}1 \right) \cdot
 \left(\sum_{t\neq 0} \nu^2(t)\right)\\
 &=\left(|\Delta(E)|-1\right) \left( \sum_{t\neq 0} \nu^2(t)\right).\end{align*} 
 It follows that
 \begin{equation}\label{squarenumber}
 |\Delta(E)|\geq 1 + \frac{(|E|^2-\nu(0))^2}{ \left(\sum_{t\in {\mathbb F}_q} \nu^2(t)\right)-\nu^2(0)}. \end{equation}
 Let us estimate $\nu(0).$ From  (\ref{zerodistance}) and  (\ref{sphereFouriertransform}), we have
 $$\nu(0)= q^4 \sum_{m\in {\mathbb F}_q^2} |\widehat{E}(m)|^2 \left( q^{-1}\delta_0(m)+q^{-3}G^2_1(\psi,\chi)\sum_{s\neq 0} \chi\left(\frac{\|m\|}{4s}\right)\right).$$
 Recall that $\widehat{E}(0,0)=q^{-2}|E|$, and observe from Theorem \ref{ExplicitGauss} that $G^2_1(\psi,\chi)=q$ for $q\equiv 1~(mod~4).$ Then we see that
 $$\nu(0)=q^{-1}|E|^2 + q^2 \sum_{m\in {\mathbb F}_q^2}|\widehat{E}(m)|^2 \sum_{s\neq 0} \chi\left(\frac{\|m\|}{4s}\right).$$
 Writing $ \sum_{m\in {\mathbb F}_q^2} = \sum_{\|m\|=0} + \sum_{\|m\|\neq 0}$ and calculating the sum over $s\neq 0,$ we see
 $$\nu(0) =q^{-1}|E|^2 +q^2 \sum_{\|m\|=0} |\widehat{E}(m)|^2(q-1) -q^2\sum_{\|m\|\neq 0} |\widehat{E}(m)|^2.$$
 Putting together the sums and applying the Plancherel theorem, we have
 \begin{equation} \label{formulazero} \nu(0)= q^{-1}|E|^2 +q^3 \sum_{\|m\|=0}|\widehat{E}(m)|^2-|E|. \end{equation}
 We now estimate $ \sum_{t\in {\mathbb F}_q} \nu^2(t).$ From Lemma \ref{main} and Lemma \ref{restriction}, we have
 \begin{align}\label{squareofv} \sum_{t\in {\mathbb F}_q} \nu^2(t)=& q^6 \left(\sum_{m\in S_0} |\widehat{E}(m)|^2\right)^2 + q^6 \sum_{t\neq 0} 
 \left(\sum_{m\in S_t}|\widehat{E}(m)|^2\right)^2 +q^{-1}|E|^4-q|E|^2 \nonumber\\
 \leq & q^6 \left(\sum_{m\in S_0} |\widehat{E}(m)|^2\right)^2+ q^6 \left( \max_{t\neq 0} \sum_{m\in S_t} |\widehat{E}(m)|^2\right)\cdot \left(\sum_{\|m\|\neq 0} |\widehat{E}(m)|^2\right)+ q^{-1}|E|^4-q|E|^2\nonumber\\
 \leq & q^6 \left(\sum_{m\in S_0} |\widehat{E}(m)|^2\right)^2+ q^6 \frac{\sqrt{3}|E|^{3/2}}{q^3}\left(q^{-2}|E|-\sum_{\|m\|= 0} |\widehat{E}(m)|^2\right)   + q^{-1}|E|^4-q|E|^2.
  \end{align}
 Letting $\Omega(E)= \sum_{\|m\|=0} |\widehat{E}(m)|^2,$ and $R(E)= q^{-1}|E|^4-q^{-2}|E|^4+2q^{-1}|E|^3+\sqrt{3}q|E|^{5/2}-q|E|^2-|E|^2$ and plugging (\ref{formulazero}) and (\ref{squareofv}) into the formula (\ref{squarenumber}), we have
 \begin{equation}\label{lowerbound}|\Delta(E)|\geq 1+ \frac{ \left(|E|^2-q^{-1}|E|^2+|E|-q^3\Omega(E)\right)^2}{(-2q^2|E|^2-\sqrt{3}q^3|E|^{3/2}+2q^3|E|) \Omega(E) +R(E)}.\end{equation}
We aim to find the lower bound of the right-hand side in (\ref{lowerbound}). 
Since $|E|\geq q^{4/3}$ and $|E|$ is a positive integer, it suffices to show that the second part of Theorem \ref{wolffin2d} holds for all $E\subset {\mathbb F}_q^2$ with $|E|=q^{\alpha}$ where $\alpha >0$ is the minimum value such that $ q^{\alpha}$ is an integer and $q^{\alpha}\geq q^{4/3}.$ The general case follows from the simple fact that $ |\Delta(E')|\leq |\Delta(E)|$ if $ E'\subset E.$ Whenever we choose such a set $E$, $\Omega(E)$ is just a constant but we don't know the exact value for $\Omega(E).$ However, the range of $\Omega(E)$ takes the following:
$$ q^{-4}|E|^2\leq \Omega(E)\leq q^{-2}|E|,$$
because  $ |\widehat{E}(0,0)|^2\leq \Omega(E)\leq \sum_{m\in {\mathbb F}_q^2}|\widehat{E}(m)|^2.$ For a fixed $E$ and $q$, we shall consider the right-hand side of (\ref{lowerbound}) as a function in terms of $\Omega(E).$ If we put $ \Omega(E)=x, a=|E|^2-q^{-1}|E|^2+|E|, b=-2q^2|E|^2-\sqrt{3}q^3|E|^{3/2}+2q^3|E|, $ and $R(E)=c$, then a lower bound of the right-hand side of (\ref{lowerbound}) is given by the minimum value of the following function:
\begin{equation}\label{function}
f(x)= \frac{(a-q^3 x)^2}{b x+c}\quad \mbox{for}\quad q^{-4}|E|^2\leq x\leq q^{-2}|E|.\end{equation}
If $|E|=q^\alpha$ is the smallest integer such that $\alpha\geq 4/3$, then we claim that the minimum of the function $f$ on $q^{-4}|E|^2\leq x\leq q^{-2}|E|$ happens at $x=q^{-4}|E|^2$ if $q$ is sufficiently large ( with the help of a calculator, $q>9$). To see this, note that $x=x_0= -b^{-1}c$ is the vertical asymptote and the critical points of the function $f$ are given by $x_1= aq^{-3}$ and $x_2=-q^{-3}b^{-1}(2q^3c+ab)$. In addition, observe that $a>0, b<0$ and $c>0.$ Thus, if $q$ is sufficiently large, then  a routine calculation shows that $ x_2\leq q^{-4}|E|^2\leq q^{-2}|E|\leq x_0\leq x_1,$ and  the local minimum and maximum happen at $x_2$ and $x_1$ respectively. Thus, our claim is justified . When we replace $\Omega(E)$  in (\ref{lowerbound}) by $ q^{-4}|E|^2$, we have
 \begin{equation}\label{goodlowerbound} |\Delta(E)|> \frac{q \left(g(E)\right)^2}{h(E)},\end{equation}
 where 
 $$ g(E)=|E|^2-2q^{-1}|E|^2+|E|$$
 and
 $$h(E)=-3q^{-1}|E|^4-\sqrt{3}|E|^{7/2}+4|E|^3 + |E|^4+\sqrt{3}q^{2}|E|^{5/2}-q^{2}|E|^2-q|E|^2.$$
 Recall that, without loss of generality, we have assumed that  the number of elements of $|E|$ is an integer $ q^{\alpha}$ where $\alpha \geq 4/3$ is the smallest real number such that $q^{\alpha}\geq q^{4/3}.$ Thus, we see that $h(E)\leq |E|^4+\sqrt{3}q^2|E|^{5/2}-\sqrt{3}|E|^{7/2} .$ Moreover, it is clear that $ g(E)\geq (1-2q^{-1})|E|^2.$ From (\ref{goodlowerbound}), it therefore follows that
 $$ |\Delta(E)|> \frac{q \left(1-2q^{-1}\right)^2}{K(|E|)},$$
 where $K(|E|)= 1 + \sqrt{3}q^2|E|^{-3/2}-\sqrt{3}|E|^{-1/2}.$
 If we consider the $K(|E|)$ as a function in terms of $|E|,$ then we can easily see that $K(|E|)\leq K(q^{4/3}),$ because $q^{4/3}\leq |E|\leq q^2$ and the function $K$ is decreasing on the interval. Thus, the proof of the second part of Theorem \ref{wolffin2d} is complete.

\section{Proof of Theorem \ref{distpinned} - Pinned distance sets}

We begin by defining the counting function,
$$\nu_y(t)=\sum_{\|x-y\|=t} E(x).$$
Squaring $\nu_y(t)$, we have
$$  \nu_y^2(t)=\sum_{\|x-y\|=\|x'-y\|=t} E(x)E(x').$$
Summing in $y\in E$ and $t\in {\mathbb F}_q$, we see
$$ \sum_{y\in E} \sum_{t\in {\mathbb F}_q} \nu_{y}^2(t) = \sum_{\|x-y\|=\|x'-y\|} E(y)E(x)E(x'),$$
applying orthogonality,
$$=q^{-1} \sum_{s\in {\mathbb F}_q} \sum_{y,x,x'\in {\mathbb F}_q^d} \chi(s(||x-y||-||x'-y||)) E(y)E(x)E(x'),$$
and extracting the $s=0$ term,

$$=q^{-1}{|E|}^3+q^{-1} \sum_{s\neq 0} \sum_{y,x,x'\in {\mathbb F}_q^d} \chi(s(||x-y||-||x'-y||)) E(y)E(x)E(x')=I+II.$$ 

Here
$$ II=q^{-1} \sum_{s \not=0} \sum_{y \in E} 
{\left| \sum_{x \in E} \chi(s(||x||-2y \cdot x)) \right|}^2,$$ since 
$$ \|x-y\|-\|x'-y\|=(||x||-2y \cdot x)-(||x'||-2y \cdot x').$$ 

It follows by extending the sum over $y \in E$ to over $y \in  {\Bbb F}_q^{d}$ that 
$$0\leq II \leq q^{-1}\sum_{s \not=0} \sum_{y \in {\Bbb F}_q^{d}} \sum_{x,x' \in E} 
\chi(-2s y \cdot (x-x')) \chi(s(||x||-||x'||)),$$
and from orthogonality in the variable $y\in {\mathbb F}_q^{d}$,
$$=q^{d-1} \sum_{s \not=0} \sum_{x\in E} 1,$$
which is less than the quantity $q^d|E|.$
It therefore follows that 
\begin{equation}\label{distancemodel}\sum_{y\in E} \sum_{t\in {\mathbb F}_q} \nu_{y}^2(t)=I + II < q^{-1}{|E|}^3 + q^d|E|.\end{equation}
Now, by the Cauchy-Schwarz inequality and above estimation, we obtain that 
$$ {|E|}^3={|E|}^{-1} \left( \sum_{y\in E}\sum_t \nu_{y}(t) \right)^2 < |E|^{-1} \sum_{y \in E} |\Delta_y(E)| \cdot (q^{-1}{|E|}^3+q^{d}{|E|}),$$ which means that 
$$ {|E|}^{-1} \sum_{y \in E}  |\Delta_y(E)|>\frac{|E|^3}{q^{-1}{|E|}^3+q^{d}{|E|}}\geq \frac{q}{2}$$
provided that $|E|\geq q^{(d+1)/2}$, which completes the proof of Theorem \ref{distpinned}.

\section{Proof of Theorem \ref{dotpinned} - Pinned dot product sets}
Here we define the function $\eta_y(s)$ by the relation 
$$ \sum_{s\in {\mathbb F}_q} g(s) \eta_y(s)=\sum_{x \in E} g(x \cdot y) E(x). $$

Taking $g(s)=q^{-1} \chi(-ts)$, we see that 
$$ \widehat{\eta}_y(t)=q^{d-1}\widehat{E}(ty).$$ 

It follows that 
$$\sum_{t\in {\mathbb F}_q} \sum_{y \in E} {|\widehat{\eta}_y(t)|}^2=q^{2(d-1)} \sum_{t\in {\mathbb F}_q} \sum_{y \in E} {|\widehat{E}(ty)|}^2,$$
and extracting $t=0$ we have that 
$$\sum_{t\in {\mathbb F}_q} \sum_{y \in E} {|\widehat{\eta}_y(t)|}^2={|E|}^3q^{-2}+q^{2(d-1)} \sum_{t \not=0} \sum_{y \in E} {|\widehat{E}(ty)|}^2,$$
which after changing variables
$$\sum_{t\in {\mathbb F}_q} \sum_{y \in E} {|\widehat{\eta}_y(t)|}^2= {|E|}^3q^{-2}+q^{2(d-1)} \sum_{x \in {\Bbb F}_q^d} {|\widehat{E}(x)|}^2 \cdot \sum_{t\neq 0} E(\frac{x}{t}).$$

Since $ \sum_{t\neq 0} E(\frac{x}{t})\leq (q-1),$ 
it follows by the Plancherel theorem that 
$$ \sum_{t\in {\mathbb F}_q} \sum_{y \in E} {|\widehat{\eta}_y(t)|}^2 \leq  {|E|}^3q^{-2}+q^{2(d-1)}(q-1)(|E|q^{-d})={|E|}^3q^{-2}+ q^{d-1}|E|-q^{d-2}|E|,$$ and applying the Plancherel theorem once again, we see that 
\begin{equation}\label{l2dot} q \sum_{t\in {\mathbb F}_q} \sum_{y \in E} {|\widehat{\eta}_y(t)|}^2= \sum_{s\in {\mathbb F}_q} \sum_{y \in E} \eta^2_y(s)
 \leq {|E|}^3q^{-1}+q^{d}{|E|}-q^{d-1}|E|.
 \end{equation}

The Cauchy-Schwarz inequality and this estimation implies that 
$$ {|E|}^3={|E|}^{-1} \left( \sum_{y\in E}\sum_{s\in {\mathbb F}_q} \eta_y(s) \right)^2 
<{|E|}^{-1} \sum_{y \in E} |\Pi_y(E)| \cdot ({|E|}^3q^{-1}+q^{d}{|E|}),$$
 which means that 
$$ {|E|}^{-1} \sum_{y \in E} |\Pi_y(E)| > \frac{q}{1+q^{d+1}{|E|}^{-2}} \geq \frac{q}{2},$$
provided that $|E|\geq q^{\frac{d+1}{2}},$ which completes the proof of Theorem \ref{dotpinned}.

\section{Proof of Theorem \ref{distpinnedproj} - Distance sets of cartesian products}
For a fixed $z\in {\mathbb F}_q$, we denote $\tilde{y}=(\pi(y),z)$ where $y\in {\mathbb F}_q^d.$
Given $\tilde{y} \in E_z,$ we define
$$\nu_{\tilde{y}}(t)=\sum_{\|x-\tilde{y}\|=t} E(x),$$
where $E_z$ was defined in Section \ref{cartprod}.
Squaring and summing in $\tilde{y}$ and $t$,
$$ \sum_{\tilde{y}\in E_z} \sum_{t\in {\mathbb F}_q} \nu_{\tilde{y}}^2(t) = \sum_{\|x-\tilde{y}\|=\|x'-\tilde{y}\|} E_z(\tilde{y})E(x)E(x'),$$
applying orthogonality,
$$=q^{-1} \sum_{s\in {\mathbb F}_q} \sum_{\tilde{y},x,x'\in {\mathbb F}_q^d} \chi(s(||x-\tilde{y}||-||x'-\tilde{y}||)) E_z(\tilde{y})E(x)E(x'),$$
and extracting the $s=0$ term,
$$=q^{-1}{|E_z|}{|E|}^2+q^{-1} \sum_{s\neq 0} \sum_{\tilde{y},x,x'\in {\mathbb F}_q^d} \chi(s(||x-\tilde{y}||-||x'-\tilde{y}||)) E_z(\tilde{y})E(x)E(x')=I+II.$$ 

Here
$$ II=q^{-1} \sum_{s \not=0} \sum_{\tilde{y} \in E_z} 
{\left| \sum_{x \in E} \chi(s(||x||-2\tilde{y} \cdot x)) \right|}^2,$$ since 
$$ \|x-\tilde{y}\|-\|x'-\tilde{y}\|=(||x||-2\tilde{y} \cdot x)-(||x'||-2\tilde{y} \cdot x').$$ 

It follows by extending the sum over $\tilde{y} \in E_z$ to over $\tilde{y} \in  {\Bbb F}_q^{d-1} \times \{z\}$ that 
$$0\leq II \leq q^{-1}\sum_{s \not=0} \sum_{\tilde{y} \in {\Bbb F}_q^{d-1} \times \{z\}} \sum_{x,x' \in E} 
\chi(-2s\tilde{y} \cdot (x-x')) \chi(s(||x||-||x'||)),$$
and from orthogonality in the variables $\pi(\tilde{y})\in {\mathbb F}_q^{d-1}$,
$$=q^{d-2} \sum_{s \not=0} \sum_{\pi(x)=\pi(x')} E(x)E(x') \chi(-2sz(x_d-x_d^\prime)) 
\chi(s(x_d^2-{x_d^\prime}^2)),$$
which may be rewritten

$$=q^{d-2} \sum_{s\in {\mathbb F}_q} \sum_{\pi(x)=\pi(x')} E(x)E(x') \chi(-2sz(x_d-x'_d)) 
\chi(s(x_d^2-{x^\prime_d}^2))-q^{d-2} \sum_{\pi(x)=\pi(x')} E(x)E(x').$$
Now since the second term is always negative,
$$< q^{d-2} \sum_{s\in {\mathbb F}_q} \sum_{\pi(x)=\pi(x')} E(x)E(x') \chi(-2sz(x_d-x'_d)) 
\chi(s(x_d^2-{x'_d}^2)).$$

Then we may apply orthogonality in $s$ to  show that this expression is equal to
$$q^{d-1}  \sum_{2z(x_d-x'_d)=x_d^2-{x'_d}^2; \pi(x)=\pi(x')} E(x)E(x'),$$
and dividing out,
$$=q^{d-1} \sum_{2z=x_d+x'_d; x_d \not=x'_d; \pi(x)=\pi(x')} E(x)E(x')+q^{d-1} \sum_{x=x'} E(x)E(x'),$$
which gives the final bound
$$II < 2q^{d-1}|E|.$$

Now, by the Cauchy-Schwarz inequality and above estimations, we obtain that 
$$ {|E|}^2|E_z|={|E_z|}^{-1} \left( \sum_{\tilde{y}\in E_z}\sum_t \nu_{\tilde{y}}(t) \right)^2 < |E_z|^{-1} \sum_{\tilde{y} \in E_z} |\Delta^{(z)}_y(E)| \cdot ({|E|}^2|E_z|q^{-1}+2q^{d-1}{|E|}),$$ which means that 
$$ {|E_z|}^{-1} \sum_{\tilde{y} \in E_z}  |\Delta^{(z)}_y(E)|> \frac{q}{1+2q^{d}{|E|}^{-1}|E_z|^{-1}}\geq \frac{q}{3},$$
provided that $|E||E_z|\geq q^d$, which completes the proof of Theorem \ref{distpinnedproj}.

\section{Proof of Theorem \ref{dotpinnedproj} - Dot product sets of cartesian products}
The proof here will follow same basic outline of the proof of Theorem \ref{distpinnedproj}.
Let $z\in {\mathbb F}_q^*.$ Given $\tilde{y}\in E_z$, we first define 
$$\nu_{\tilde{y}}(t)=\sum_{x\cdot \tilde{y}=t} E(x).$$
Then $\nu_{\tilde{y}}^2(t)=\sum_{x\cdot\tilde{y}=x\cdot\tilde{y}=t} E(x)E(x').$
If we sum in $\tilde{y}\in E_z \subset {\mathbb F}_q^d$ and $t\in {\mathbb F}_q,$ then we have
$$\sum_{\tilde{y} \in E_z} \sum_{t\in {\mathbb F}_q} \nu_{\tilde{y}}^2(t) = \sum_{x\cdot \tilde{y}=x'\cdot \tilde{y}} E_z(\tilde{y})E(x)E(x').$$
Then applying orthogonality in $s\in {\mathbb F}_q$ and extracting $s=0$, we obtain

$$\sum_{\tilde{y} \in E_z} \sum_{t\in {\mathbb F}_q} \nu_{\tilde{y}}^2(t) = |E_z||E|^2q^{-1}+
q^{-1}\sum_{s\neq 0}\sum_{\substack{\tilde{y} \in E_z \\ x,x' \in E}} \chi(s\tilde{y}\cdot(x-x'))
=I+II.$$

Now, for $z \in {\mathbb F}^*_q,$ we have
$$II = q^{-1}\sum_{s\neq 0}\sum_{\tilde{y} \in E_z}\left| \sum_{x\in E} \chi(sx\cdot \tilde{y})\right|^2,$$
and by extending the sum over $ E_z$ to over ${\mathbb F}_q^{d-1}$ we have that this quantity is
$$\leq q^{-1}\sum_{s\neq 0}
\sum_{\tilde{y} \in \mathbb F_q^{d-1}\times \{z\} }\left| \sum_{x\in E} \chi(sx\cdot \tilde{y})\right|^2.$$
Then applying orthogonality in the variable $\pi(\tilde{y})$ we have that
$$II \leq q^{d-2}\sum_{s\neq 0}
\sum_{\pi(x)=\pi(x')} \chi(sz(x_d-x_d'))E(x)E(x'),$$
and extracting the term $x_d=x_d'$ gives
$$=q^{d-2}\sum_{s\neq 0}
\sum_{\pi(x)=\pi(x'); x_d=x'_d} E(x)E(x')$$
$$+q^{d-2}\sum_{s\neq 0}\sum_{\pi(x)=\pi(x'); x_d\neq x'_d} \chi(sz(x_d-x_d'))E(x)E(x')$$
$$=q^{d-2}(q-1)|E|-q^{d-2}\sum_{\pi(x)=\pi(x'); x_d\neq x'_d}E(x)E(x')$$
$$< q^{d-1}|E|.$$
Thus it follows that for each $z \in {\mathbb F}_q^*,$
$$ \sum_{\tilde{y} \in E_z} \sum_{t\in {\mathbb F}_q} \nu_{\tilde{y}}^2(t) < |E_z||E|^2q^{-1}+ q^{d-1}|E|.$$
Now, by the Cauchy-Schwarz inequality and this estimation, we have 
$$ {|E|}^2|E_z|={|E_z|}^{-1} \left( \sum_{\tilde{y}\in E_z}\sum_{t\in {\mathbb F}_q} \nu_{\tilde{y}}(t) \right)^2 < |E_z|^{-1} \sum_{\tilde{y} \in E_z} |\Pi^{(z)}_y(E)| \cdot ({|E|}^2|E_z|q^{-1}+q^{d-1}{|E|}),$$ which means that 
$$ {|E_z|}^{-1} \sum_{\tilde{y} \in E_z}  |\Pi^{(z)}_y(E)|> \frac{q}{1+q^{d}{|E|}^{-1}{|E_z|}^{-1}}\geq \frac{q}{2},$$
provided that $|E||E_z|\geq q^d$. Thus the proof of Theorem \ref{dotpinnedproj} is complete.

\section{Proof of Theorem \ref{distmultproj} - $k$-star distance sets}

We begin by defining the counting function,
$$\nu_{y^1,\dots,y^{k}}(t_1,\dots,t_{k})=\mathop{\sum}_{\|x-y^1\|=t_1,\dots, \|x-y^{k}\|=t_{k}} E(x).$$
The proof of Theorem \ref{distmultproj}  is based on the following lemma.

\begin{lemma} \label{B} Let $E \subset {\Bbb F}_q^d$. Then
$$\sum_{y^1,\dots,y^k\in E}\sum_{t_1, t_2,\dots,t_k \in {\mathbb F}_q}|\nu_{y^1, y^2,\dots,y^k}(t_1, t_2,\dots,t_k)|^2 \lesssim \frac{|E|^{k+2}}{q^{k}}+q^d|E|^{k}.$$
\end{lemma}
\begin {proof}
We proceed by induction.  The initial case follows from the estimation (\ref{distancemodel}).  Suppose that 
$$\sum_{y^1,\dots,y^{k-1}\in E} \sum_{t_1,\dots,t_{k-1}\in {\mathbb F}_q} \nu_{y^1,\dots,y^{k-1}}^2(t_1,\dots,t_{k-1})\lesssim \frac{|E|^{k+1}}{q^{k-1}}+  q^{d}|E|^{k-1}.$$

Now
 $$\sum_{y^1,\dots,y^{k-1},y^k \in E } \sum_{t_1,\dots,t_{k}\in {\mathbb F}_q} \nu_{y^1,\dots,y^{k-1}, y^k}^2(t_1,\dots,t_{k})= $$
$$\mathop{\sum\dots \sum}_{\substack{\|x-y^1\|=\|x'-y^1\|,\dots,\|x-y^{k-1}\|=\|x'-y^{k-1}\| \\ \|x- y^{k}\|=\|x'- y^{k}\|}} E(y^1)\dots E(y^{k-1})  
E( y^{k}) E(x)E(x').$$

Then applying orthogonality,
$$= q^{-1} \sum_{s \in \mathbb F_q}\mathop{\sum\dots \sum}_{\substack{\|x-y^1\|=\|x'-y^1\|,\dots,\|x-y^{k-1}\|=\|x'-y^{k-1}\| \\ x,x',y^1,\dots, y^{k-1}, y^k \in E }} \chi(s(||x||-2 y^{k} \cdot x))\chi(-s(||x'||-2 y^{k} \cdot x')).$$
since 
$$ \|x- y^{k}\|-\|x'- y^{k}\|=(||x||-2 y^{k} \cdot x)-(||x'||-2 y^{k} \cdot x').$$ 

Extracting the $s=0$ term and applying the induction hypothesis gives
$$\lesssim \frac{|E|^{k+2}}{q^{k}}+  q^{d-1}|E|^{k}+R,$$
where 
$$R= q^{-1} \sum_{s \in \mathbb F^*_q}\mathop{\sum\dots \sum}_{\substack{\|x-y^1\|=\|x'-y^1\|,\dots,\|x-y^{k-1}\|=\|x'-y^{k-1}\| \\ x,x',y^1,\dots, y^{k-1},y^k\in E }} \chi(s(||x||-2 y^{k} \cdot x))\chi(-s(||x'||-2 y^{k} \cdot x')).$$

Then $R$ may be expressed as
$$ q^{-1} \sum_{s \in \mathbb F_q^*} \sum_{t_1,\dots,t_{k-1}\in {\mathbb F}_q} \sum_{\substack{y^1,\dots y^{k-1} \in E \\  y^{k}\in E}} 
{\left|\mathop{\sum\dots \sum}_{\substack{\|x-y^1\|=t_1,\dots,\|x-y^{k-1}\|=t_{k-1} \\ x\in E }} \chi(s(||x||-2 y^{k} \cdot x)) \right|}^2.$$
Then extending sum over $ y^{k} \in E$ to over $ y ^{k}\in  {\Bbb F}_q^{d}$, expanding the square, and applying orthogonality in $y^{k}$ gives
$$R\leq q^{d-1} \sum_{s \in \mathbb F_q^*} \sum_{y^1,\dots y^{k-1}, x \in E} 1$$

which in turn is less than $ q^d |E|^k.$

Therefore we have
$$\sum_{y^1,\dots,y^{k}\in E} \sum_{t_1,\dots,t_{k}\in {\mathbb F}_q} \nu_{y^1,\dots,y^{k}}^2(t_1,\dots,t_{k})\lesssim \frac{|E|^{k+2}}{q^{k}}+  q^{d}|E|^{k},$$
which completes the proof of Lemma \ref{B}.
\end{proof}

We are ready to complete the proof of Theorem \ref{distmultproj}.
By the Cauchy-Schwarz inequality, we have

$$|E|^{2k+2}=\left( \sum_{y^1,\dots,y^{k}\in E}\sum_{t_1, t_2,\dots,t_{k}
\in {\Bbb F}_q}
\nu_{y^1, y^2,\dots,y^{k}}(t_1, t_2,\dots,t_{k})\right)^2$$

$$\leq \sum_{y^1,\dots,y^{k}\in E}|\Delta_{y^1, y^2,\dots,y^{k}}(E)| \cdot
\sum_{y^1,\dots,y^{k}\in E}\sum_{t_1, t_2,\dots,t_{k} \in {\Bbb
F}_q}|\nu_{y^1, y^2,\dots,y^{k}}(t_1, t_2,\dots,t_{k})|^2.$$
By Lemma \ref{B} it follows that
$$|E|^{2k+2} \lesssim \sum_{y^1,\dots,y^{k}\in E}|\Delta_{y^1,y^2,\dots,y^{k}}(E)| \cdot \left( \frac{|E|^{k+2}}{q^{k}}+q^d|E|^{k}\right).$$
Therefore,
$$\sum_{y^1,\dots,y^{k}\in E}|\Delta_{y^1, y^2,\dots,y^{k}}(E)|\gtrsim \frac{|E|^{2k+2}}{\frac{|E|^{k+2}}{q^{k}}+q^d|E|^{k}}.$$
Normalize to obtain
$$\frac{1}{|E|^{k}}\sum_{y^1,\dots,y^{k}\in E}|\Delta_{y^1, y^2,\dots,y^{k}}(E)| \gtrsim \frac{|E|^{k+2}}{\frac{|E|^{k+2}}{q^{k}}+q^d|E|^{k}}, $$
which for $|E| \gtrsim q^{\frac{d+k}{2}}$ gives 
$$\frac{1}{|E|^{k}}\sum_{y^1,\dots,y^{k}\in E}|\Delta_{y^1, y^2,\dots,y^{k}}(E)| \gtrsim  q^{k} .$$
Thus the proof of Theorem \ref{distmultproj}  is complete.

\section{Proof of Theorem \ref{kpoint} - $k$-simplices}
If $k=1$, then the statement of Theorem \ref{kpoint} immediately follows from Theorem \ref{distpinned}. We therefore assume that $k\geq 2.$
As stated in the introduction in order to specify a $k$-simplex up to isometry it is enough to specify the distances determined by the points.
Here we will specify our $k$-simplices using Theorem \ref{distmultproj} as one set of distances at a time.
In addition, we need the following theorem which is more general version of Theorem \ref{distmultproj}.
\begin{theorem}\label{Alexthm} Given $E\subset {\mathbb F}_q^d,$ let ${\mathcal E} \subset E\times \cdots \times E= E^s ,s\geq 2,$ with $|{\mathcal E}|\sim |E|^s.$
Define 
$${\mathcal E'}=\{(y^1,\dots,y^{s-1})\in E^{s-1}: (y^1,\dots,y^{s-1},y^s)\in {\mathcal E} ~~\text{for some}~~y^s\in E\}.$$
In addition, for each $(y^1,\dots,y^{s-1})\in {\mathcal E'}$ we define 
$$ {\mathcal E}(y^1,\dots,y^{s-1})=\{y^s\in E: (y^1,\dots,y^{s-1},y^s)\in {\mathcal E}\}.$$
If $|E|\gtrsim q^{\frac{d+s-1}{2}}$, then we have
\begin{equation}\label{Alexindex}\frac{1}{|{\mathcal E}'|} \sum_{(y^1,\dots,y^{s-1})\in {\mathcal E}'} \left| \Delta_{y^1,\dots,y^{s-1}}\left({\mathcal E}(y^1,\dots,y^{s-1})\right)\right|\gtrsim q^{s-1},\end{equation}
where 
$$\Delta_{y^1,\dots,y^{s-1}}\left({\mathcal E}(y^1,\dots,y^{s-1})\right)=\{ \left(\|y^s-y^1\|,\dots, \|y^s-y^{s-1}\|\right)\in ({\mathbb F}_q)^{s-1}: y^s\in {\mathcal E}(y^1,\dots,y^{s-1})\}.$$
\end{theorem}

\begin{proof} For each $t_1,\dots,t_s \in {\mathbb F}_q, $ the incidence function on $\Delta_{y^1,\dots,y^{s-1}}({\mathcal E}(y^1,\dots,y^{s-1}))$ is given by 
$$ \nu_{y^1,\dots,y^{s-1}}^{{\mathcal E}(y^1,\dots,y^{s-1})}(t_1,\dots,t_{s-1})= |\{y^s \in {\mathcal E}(y^1,\dots, y^{s-1}): \|y^s-y^1\|=t_1, \dots, \|y^s-y^{s-1}\|=t_{s-1}\}|.$$
Observe that
$$ \nu_{y^1,\dots,y^{s-1}}^{{\mathcal E}(y^1,\dots,y^{s-1})}(t_1,\dots, t_{s-1}) \leq \nu_{y^1,\dots,y^{s-1}}(t_1,\dots, t_{s-1})=|\{(y^s \in E: ||y^s-y^1||=t_1,\dots,
||y^s-y^{s-1}||=t_{s-1}\}|.$$

By the Cauchy-Schwarz inequality, we have

$$\left|{\mathcal E}\right|^{2}=\left( \sum_{(y^1,\dots,y^{s-1})\in {\mathcal E'}} \sum_{t_1\dots,t_{s-1}\in {\mathbb F}_q}
\nu_{y^1,\dots,y^{s-1}}^{{\mathcal E}(y^1,\dots,y^{s-1})}(t_1,\dots, t_{s-1}) \right)^2$$\small

$$\leq \left(\sum_{(y^1,\dots,y^{s-1})\in {\mathcal E'}}|\Delta_{y^1,\dots,y^{s-1}}({\mathcal E}(y^1,\dots,y^{s-1}))|\right) \cdot \left( \sum_{y^1,\dots,y^{s-1}\in E}\sum_{t_1,\dots,t_{s-1} \in {\Bbb F}_q}|\nu_{ y^1,\dots,y^{s-1}}(t_1,\dots,t_{s-1})|^2\right).$$\small

Using Lemma \ref{B}, we therefore have
$$\left|{\mathcal E}\right|^{2}\leq \sum_{(y^1,\dots,y^{s-1})\in {\mathcal E'}}|\Delta_{y^1,\dots,y^{s-1}}\left({\mathcal E}(y^1,\dots,y^{s-1})\right)|  \cdot \left( \frac{|E|^{s+1}}{q^{s-1}}+q^d|E|^{s-1}\right).$$\small
Observe that $|{\mathcal E}'|\sim |E|^{s-1}$ because otherwise $|{\mathcal E}|\leq |{\mathcal E}'||E|<< |E|^s$ which contradicts $|{\mathcal E}|\sim |E|^s.$
Therefore, if $|E|\gtrsim q^{(d+s-1)/2}$, then 
it follows that 
$$\frac{1}{|{\mathcal E'}|} \sum_{(y^1,\dots,y^{s-1})\in {\mathcal E'}} \left| \Delta_{y^1,\dots,y^{s-1}}\left({\mathcal E}(y^1,\dots,y^{s-1})\right)\right|\gtrsim q^{s-1}.$$
Thus the proof of Theorem \ref{Alexthm} is complete.
\end{proof}
When a pigeon-holing argument is applied to the inequality (\ref{Alexindex}) in Theorem \ref{Alexthm}, the following corollary immediately follows.
\begin{corollary}\label{winner}
Let $E\subset {\mathbb F}_q^d$ and ${\mathcal E}\subset E\times\cdots\times E= E^s, s\geq 2,$ with $|{\mathcal E}|\sim |E|^s.$
If $|E|\gtrsim q^{\frac{d+s-1}{2}}$, then there exists $ {\mathcal E}^{(1)} \subset {\mathcal E}'\subset E^{s-1}$ with $|{\mathcal E}^{(1)}|\sim |{\mathcal E}'|\sim |E|^{s-1}$
such that for every $(y^1,\dots,y^{s-1})\in {\mathcal E}^{(1)},$ 
$$ \left|\Delta_{y^1,\dots,y^{s-1}}({\mathcal E}(y^1,\dots,y^{s-1}))\right|\gtrsim q^{s-1}.$$
Namely, the elements in $ {\mathcal E}$ determines a positive proportion of all $(s-1)$-simplices whose bases are fixed as a $(s-2)$-simplex given by any  element $(y^1,\dots,y^{s-1})\in {\mathcal E}^{(1)}.$
\end{corollary}

We are now ready to prove Theorem \ref{kpoint}.
First, using a pigeon-holing argument together with Theorem \ref{distmultproj}, we see that for $|E| \gtrsim q^{\frac{d+k}{2}},$ there exists a set ${\mathcal E} \subset E\times \dots \times E=E^k$ with $|{\mathcal E}|\gtrsim |E|^k$ such that for every $(y^1,\dots,y^k)\in {\mathcal E},$ we have
$$|\Delta_{y^1,\dots,y^k}(E)|=|\{(\|y^0-y^j\|)_{1\leq j\leq k} \in (\mathbb F_q)^{k}:y^0\in E\}|\gtrsim q^{k}.$$
Notice that this  implies that if $|E|\gtrsim q^{\frac{d+k}{2}}$, then the set $E$ determines a positive proportion of all $k$-simplices whose bases are given by any fixed $(k-1)-$simplex determined by $(y^1,\dots,y^k)\in {\mathcal E}.$ It therefore suffices to show that a positive proportion of all $(k-1)$-simplices can be constructed by the elements of ${\mathcal E}.$ Since $|{ E}| \gtrsim q^{\frac{d+k}{2}}\gtrsim q^{\frac{d+k-1}{2}}$ and $|{\mathcal E}|\sim |E|^k,$ we can apply Corollary \ref{winner} where  $s$ is replaced by $k.$ Then we see that there exists  a set ${\mathcal E}^{(1)}\subset {\mathcal E}' $ with $|{\mathcal E}^{(1)}|\sim |{\mathcal E}'|\sim |E|^{k-1}$
such that for every $(y^1,\dots,y^{k-1})\in {\mathcal E}^{(1)},$ we have
$$\left|\Delta_{y^1,\dots,y^{k-1}}( {\mathcal E}(y^1,\dots,y^{k-1}))\right|\gtrsim q^{k-1}.$$
Observe that this estimation implies that the elements in ${\mathcal E}$ determines a positive proportion of all possible $(k-1)$-simplices where their bases are fixed by a $(k-2)-$simplex given by any $ (y^1,\dots,y^{k-1})\in {\mathcal E}^{(1)}.$ Thus, it is enough to show that the elements in ${\mathcal E}^{(1)}$ can determine a positive proportion of all $(k-2)-$simplices. Putting ${\mathcal E}^{(0)}={\mathcal E}$ and using Corollary \ref{winner}, if we repeat above process $p$-times, then we see that there exists a set $ {\mathcal E}^{(p)}\subset \left({\mathcal E}^{(p-1)}\right)'\subset E^{k-p}$ with $|{\mathcal E}^{(p)}|\sim | \left({\mathcal E}^{(p-1)}\right)'|\sim |E|^{k-p}$ such that 
for each $(y^1,\dots,y^{k-p})\in {\mathcal E}^{(p)},$ we have 
$$ \left| \Delta_{y^1,\dots,y^{k-p}}( {\mathcal E}^{(p-1)}(y^1,\dots,y^{k-p}))\right|\gtrsim q^{k-p},$$
and so it suffices to show that the elements in ${\mathcal E}^{(p)}\subset E^{k-p}$ determine a positive proportion of all $(k-p-1)-$simplices.
Taking $p=k-2$, we reduce our problem to showing that the elements in $ {\mathcal E}^{(k-2)}\subset E\times E$ determine a positive proportion of all $1-$simplices. However, it is clear by applying Corollary \ref{winner} after setting $s=2, {\mathcal E}={\mathcal E}^{(k-2)}.$ To see this, first notice from our repeated process that ${\mathcal E}^{(k-2)}\subset E\times E$ and $|{\mathcal E}^{(k-2)}|\sim |E|^2.$ Since $|E|\gtrsim q^{\frac{d+k}{2}}\gtrsim q^{\frac{d+1}{2}},$ Corollary \ref{winner} yields the desirable result. Therefore, we complete the proof of Theorem \ref{kpoint}.

\section{Proof of Theorem \ref{dotmultproj} - $k$-star dot product sets }
Define
$\eta_{y^1, y^2,\dots,y^{k}}(s_1, s_2,\dots,s_k)$ by the relation
$$\sum_{s_1, s_2,\dots,s_k \in {\Bbb F}_q}g(s_1,
s_2,\dots,s_k)\eta_{y^1, y^2,\dots,y^k}
(s_1, s_2,\dots,s_k)=\sum_{x \in {\Bbb F}_q^d}g(x \cdot y^1,x \cdot
y^2,\dots,x \cdot y^k)E(x),$$
where $g$ is a complex-valued function on ${\mathbb F}_q^{k},$ and $y^j\in {\mathbb F}_q^d$ for $j=1,2,\dots ,k.$
The proof of Theorem \ref{dotmultproj}  is based on the following lemma.
\begin{lemma} \label{A} Let $E \subset {\Bbb F}_q^d$. Then
$$\sum_{y^1,\dots,y^k\in E}\sum_{s_1, s_2,\dots,s_k \in {\Bbb
F}_q}|\eta_{y^1, y^2,\dots,y^k}(s_1, s_2,\dots,s_k)|^2
\lesssim \frac{|E|^{k+2}}{q^{k}}+q^d|E|^{k}.$$
\end{lemma}

\begin {proof}
We proceed by induction.  The initial case follows from equation (\ref{l2dot}).  Suppose that 
$$\sum_{y^1,\dots,y^{k-1}\in E}\sum_{s_1, s_2,\dots,s_{k-1} \in {\mathbb
F}_q}|\eta_{y^1, y^2,\dots,y^{k-1}}(s_1, s_2,\dots,s_{k-1})|^2
\lesssim \frac{|E|^{k+1}}{q^{k-1}}+q^d|E|^{k-1}.$$

Let $g(s_1, s_2,\dots,s_k)=q^{-k} \chi(-s_1t_1-s_2t_2-\dots-s_kt_k)$. 
 It follows that 
$$\widehat{\eta}_{y^1, y^2,\dots,y^k}(t_1,t_2,\dots,t_k)=q^{d-k}\widehat{E}(t_1y^1+t_2y^2+\dots+t_ky^k).$$
Then substituting in,
$$\sum_{t_1,\dots,t_k \in {\Bbb F}_q} \sum_{y^1,\dots,y^k \in E}|\widehat{\eta}_{y^1, y^2,\dots,y^k}(t_1, t_2,\dots,t_k)|^2$$
$$= q^{2(d-k)}\sum_{t_1,\dots,t_k \in {\Bbb F}_q}\sum_{y^1,\dots,y^k \in E}|\widehat{E}(t_1y^1+t_2y^2+\dots+t_ky^k)|^2,$$
and extracting the case when $t_k=0$ we have
$$q^{2(d-k)}|E|\sum_{t_1,\dots,t_{k-1}\in {\mathbb F}_q} \sum_{y^1,\dots,y^{k-1} \in E}|\widehat{E}(t_1y^1+t_2y^2+\dots+t_{k-1}y^{k-1})|^2 $$
$$+ q^{2(d-k)}\sum_{\substack{t_1,\dots,t_{k-1}\in \mathbb F_q \\t_k\neq 0}} \sum_{y^1,\dots,y^k \in E}|\widehat{E}(t_1y^1+t_2y^2+\dots+t_ky^k)|^2=I+II$$

For the first term we apply Plancherel and the induction hypothesis to get
$$I\lesssim \frac{|E|^{k+2}}{q^{2k}}+q^{d-k-1}|E|^k.$$

For the second term we write,
\begin{align*}
II &= q^{2(d-k)}\sum_{\substack{t_1,\dots,t_{k-1}\in \mathbb F_q \\t_k\neq 0}}\sum_{y^1,\dots,y^k \in E}|\widehat{E}(t_1y^1+t_2y^2+\dots+t_ky^k)|^2 \\
&= q^{2(d-k)} \sum_{y^1,\dots,y^{k-1} \in E}\sum_{\substack{t_1,\dots,t_{k-1}\in \mathbb F_q \\t_k\neq 0}} \left( \sum_{y^k \in {\mathbb F}_q^d } E(y^k)|\widehat{E}(t_1y^1+t_2y^2+\dots+t_ky^k)|^2\right), \\
\end{align*}
and changing variables gives
$$\lesssim q^{2(d-k)} \sum_{y^1,\dots,y^{k-1} \in E}\sum_{\substack{t_1,\dots,t_{k-1}\in \mathbb F_q \\t_k\neq 0}} \sum_{m\in {\mathbb F}_q^d}|\widehat{E}(m)|^2E(t_1y^1+\dots t_{k-1}y^{k-1}+mt_k^{-1}), $$
which summing in $t_1,\dots,t_{k}$ gives
\begin{equation}\label{hyper}
=q^{2(d-k)} \sum_{y^1,\dots,y^{k-1} \in E}\sum_{m\in {\mathbb F}_q^d}|\widehat{E}(m)|^2|E\cap H_{y^1,\dots,y^{k-1},m}|, 
\end{equation}
where  $H_{y^1,\dots,y^{k-1},m}$ is $k$ dimensional hyperplane running through the origin.
Since $|E\cap H_{y^1,\dots,y^{k-1}}|\leq q^{k-1}$,
$$\lesssim q^{2(d-k)}|E|^{k-1}q^{k}\sum_{m \in {\Bbb F}_q^d}|\widehat{E}(m)|^2= q^{d-k}|E|^{k}. $$

Therefore we have that

$$\sum_{t_1,\dots,t_k \in {\Bbb F}_q} \sum_{y^1,\dots,y^k \in
E}|\widehat{\eta}_{y^1, y^2,\dots,y^k} 
(t_1, t_2,\dots,t_k)|^2 \lesssim \frac{|E|^{k+2}}{q^{2k}}+q^{d-k}|E|^{k}.$$

Applying Plancherel in $t_1, \dots, t_k$ we obtain

$$\sum_{y^1,\dots,y^k\in E}\sum_{s_1, s_2,\dots,s_k \in {\Bbb
F}_q}|\eta_{y^1, y^2,\dots,y^k}(s_1, s_2,\dots,s_k)|^2
\lesssim \frac{|E|^{k+2}}{q^{k}}+q^d|E|^{k}.$$

\end {proof}

We are ready to complete the proof of Theorem \ref{dotmultproj} .
By the Cauchy-Schwarz inequality, we have

$$|E|^{2(k+1)}=\left( \sum_{y^1,\dots,y^k\in E}\sum_{s_1, s_2,\dots,s_k
\in {\Bbb F}_q}
\eta_{y^1, y^2,\dots,y^k}(s_1, s_2,\dots,s_k)\right)^2$$

$$\lesssim \sum_{y^1,\dots,y^k\in E}|\Pi_{y^1, y^2,\dots,y^k}(E)| \cdot
\sum_{y^1,\dots,y^k\in E}\sum_{s_1, s_2,\dots,s_k \in {\Bbb
F}_q}|\eta_{y^1, y^2,\dots,y^k}(s_1, s_2,\dots,s_k)|^2.$$
By Lemma \ref{A} it follows that
$$|E|^{2k+2} \lesssim \sum_{y^1,\dots,y^k\in E}|\Pi_{y^1,
y^2,\dots,y^k}(E)| \cdot
\left( \frac{|E|^{k+2}}{q^{k}}+q^d|E|^{k}\right).$$
Therefore,
$$\sum_{y^1,\dots,y^k\in E}|\Pi_{y^1, y^2,\dots,y^k}(E)|
\gtrsim \frac{|E|^{2k+2}}{\frac{|E|^{k+2}}{q^{k}}+q^d|E|^{k}}.$$
Normalize to obtain
$$\frac{1}{|E|^{k}}\sum_{y^1,\dots,y^k\in E}|\Pi_{y^1,
y^2,\dots,y^k}(E)| \gtrsim \frac{|E|^{k+2}}{\frac{|E|^{k+2}}{q^{k}}+q^d|E|^{k}}, $$
which for $|E| \gtrsim q^{\frac{d+k}{2}}$ gives 
$$\frac{1}{|E|^{k}}\sum_{y^1,\dots,y^k\in E}|\Pi_{y^1, y^2,\dots,y^k}(E)| \gtrsim  q^{k} .$$
Thus the proof of Theorem \ref{dotmultproj}  is complete.

\section{Proof of Theorem \ref{distmultprojsphere} - $k$-star distance sets on a sphere }
Here we only need to prove the following lemma whose proof we will briefly sketch.
\begin{lemma}  Let $E \subset S$. Then
$$\sum_{y^1,\dots,y^k\in E}\sum_{s_1, s_2,\dots,s_k \in {\Bbb
F}_q}|\nu_{y^1, y^2,\dots,y^k}(s_1, s_2,\dots,s_k)|^2
\lesssim \frac{|E|^{k+2}}{q^{k}}+q^{d-1}|E|^{k}.$$
\end{lemma}

Since $E$ is a subset of a sphere counting distances is equivalent to dot products. Therefore we return to the proof of Lemma \ref{A}. 
Recall the equation (\ref{hyper}) is specifically given by
$q^{2(d-k)} \sum_{y^1,\dots,y^{k-1} \in E}\sum_{m}|\widehat{E}(m)|^2|E\cap H_{y^1,\dots,y^{k-1},m}|.$
Since $E$ is a subset of a sphere, we see that $|E\cap H_{y^1,\dots,y^{k-1},m}|\lesssim q^{k-1}$.  The rest of the proof is similar to the proof of Theorem \ref{dotmultproj}.

\section{Proof of Lemma \ref{sphere}: Gauss sums and the sphere}

Let $\chi$ be a canonical additive character of ${\mathbb F}_q$ and $\psi$ a quadratic character of 
${\mathbb F}_q.$ Recall that $\psi(0)=0, \psi(t)=1$ if $t$ is a square in ${\mathbb F}_q,$ and $\psi(t)=-1 $ if $t$ is not a square number in ${\mathbb F}_q.$
For each $a\in {\mathbb{F}_q}$, the Gauss sum $G_a(\psi, \chi)$ is defined by 
$$ G_a(\psi, \chi) = \sum_{s\in {\mathbb F}_q^*} \psi(s) \chi(as).$$ The magnitude of the Gauss sum is given by the relation
$$ |G_a(\psi, \chi)| = \left\{\begin{array}{ll} q^{\frac{1}{2}} \quad &\mbox{if} \quad a\ne 0\\
                                                  0 \quad & \mbox{if} \quad a=0. \end{array}\right.$$

We appeal to the following expression (see Theorem $6.26$ and Theorem $6.27$ in \cite{LN97}):
$$ |S_t|=\left\{\begin{array}{ll} q^{d-1}+ q^{(d-1)/2} \psi\left( (-1)^{\frac{d-1}{2}}t\right)
\quad &\mbox{if} ~ d ~\mbox{is odd}\\
 q^{d-1}+ \mu(t) q^{\frac{d-2}{2}} \psi\left( (-1)^{\frac{d}{2}} \right)
 \quad &\mbox{if} ~ d ~\mbox{is even},\end{array}\right.$$
 where $\mu(t)=q-1$ if $t=0$, and $\mu(t)=-1$ if $t\in {\mathbb F}_q^*.$
 
We also need the following estimate of the Fourier transform of spheres (see \cite{IKo08}): for each $m \not=(0,\ldots,0)$, we have 
\begin{align*}\widehat{S}_t(m)=&q^{-d-1}\psi^d(-1) (G_1(\psi,\chi))^d\sum_{s \not=0} \chi \left( \frac{||m||}{4s}+st \right) \psi^d(s).\end{align*} 

Using the explicit formula for $|S_t|$ and observing that $\sum_{t\in {\mathbb F}_q} \mu(t)=0=\sum_{t\in {\mathbb F}_q} \psi(t),~ 
 \sum_{t\in {\mathbb F}_q} \mu^2(t)=(q-1)^2+(q-1)$, 
 and $\sum_{t\in {\mathbb F}_q} \psi^2\left( (-1)^{\frac{d-1}{2}}t\right)=(q-1),$ we can easily see that
$$ \sum_{t\in {\mathbb F}_q} |S_t|^2= q^{2d-1} +q^{d}-q^{d-1}, $$ which proves the first formula in Lemma \ref{sphere}.

 For $m\neq (0, \cdots, 0)$, apply  orthogonality in $t$, and then we have
 $$\sum_{t\in {\mathbb F}_q} |\widehat{S_t}(m)|^2=q^{-d-2} \sum_{t\in {\mathbb F}_q} \sum_{s,s'\neq 0}\chi \left(\frac{\|m\|}{4} 
 (\frac{1}{s}-\frac{1}{s'})\right) \chi(t(s-s'))\psi^d(s s'^{-1})= q^{-d}-q^{-d-1},$$ which completes the proof of the second formula in \ref{sphere}.
 
Finally, again from orthogonality in $t$, we see  
$$\sum_{t\in {\mathbb F}_q} |S_t| \widehat{S}_t(m)
=\left\{ \begin{array}{ll} q^{\frac{-d-3}{2}} \psi\left( (-1)^{\frac{d+1}{2}}\right)
(G_1(\psi,\chi))^d \sum\limits_{s\neq 0} \psi(s) \chi\left(\frac{\|m\|}{4s}\right)
\sum\limits_{t\in \mathbb{F}_q} \psi(t) \chi(st) ~&\mbox{for}~ d ~\mbox{odd}\\
q^{\frac{-d-4}{2}} \psi\left( (-1)^{\frac{d}{2}}\right)(G_1(\psi,\chi))^d
\sum\limits_{s\neq 0}\chi\left(\frac{\|m\|}{4s}\right)\sum\limits_{t\in {\mathbb F}_q}\mu(t)\chi(st),
~&\mbox{for}~ d ~\mbox{even} \end{array}\right.$$
where we used that for each $s\neq 0$,  $\psi^d(s)=\psi(s)$ for $d$ odd, and  $\psi^d(s)=1$ for $d$ even.
Since $\sum_{t\in \mathbb{F}_q} \psi(t) \chi(st)= \psi(s^{-1}) G_1(\psi,\chi)$ and 
$\sum_{t\in {\mathbb F}_q}\mu(t)\chi(st)=q$ for each $s\neq 0$, using the estimation of Gauss sums, we conclude that
$$  \sum_t |S_t| \widehat{S}_t(m) \leq 1-q^{-1},$$
where we also used that $ \sum\limits_{r\neq 0} \chi\left(\frac{\|m\|}{4r}\right) \leq (q-1).$ 
Thus the proof of Lemma \ref{sphere} is complete.

%\section{Comments}

%As is the case with most of the results contained in this article one may with only minor notational changes to the proof  give more general results.  This is illustrated by the following example.  
%\begin{theorem}
%Let $A_1,\dots,A_{d-1}\subset \mathbb F_q$, $B_1,\dots ,B_d\subset \mathbb F_q$ and $z\in \mathbb F^*_q$.  If $|B_d|\prod_{j=1}^{d-1} |A_j||B_j|\geq q^{d}$ then there exist subsets $A'_1,\dots A'_{d-1}$ of $A_j$'s with $|A'_j|\gtrsim |A_j|$ such that 
%for any $a_j \in A'_j,j=1,\dots, d-1$,
%$$|a_1B_1+a_2B_2+\dots +a_{d-1}B_{d-1}+zB_d| > \frac{q}{2}.$$
%\end{theorem}

%In the continuous Euclidean setting this is analogous to the fact that for a set $E$ with
%Hausdorff dimension greater than $\frac{d+1}{2}$ the distance set is not only of positive measure but contains an interval (\cite{MS99}).


\begin{thebibliography}{7} 

\bibitem{AS02} P. Agarwal and M. Sharir, {\it The number of congruent simplexes in a point set}, Discr. Comp. Geom. \textbf{28}, 123-150, (2002). 

\bibitem{AAPS07} P. Agarwal. R. Apfelbaum, G. Purdy and M. Sharir, {\it Similar simplices in d-dimensional point set}, (preprint), (2007).

\bibitem{Bo05}  J. Bourgain, {\it Mordell's exponential sum estimate
revisited}, J. Amer. Math. Soc. \textbf{18} (2005), no. 2, 477-499.

\bibitem{BGK06} J. Bourgain, A. Glibichuk, and S. Konyagin {\it Estimates for the number of sums and products for exponentials sums in fields of prime order}, Jour. of London Math. Soc. \textbf{73} (2006) 380-398.

\bibitem{BKT04} J. Bourgain, N. Katz, and T. Tao, {\it A sum-product estimate in finite fields, and applications} Geom. Funct. Anal. \textbf{14} (2004), 27-57.

\bibitem{B86} J. Bourgain, {\it A SzemerŽdi type theorem for sets of positive density}, Israel J. Math. \textbf{54} (1986), no. 3, 307-331.

\bibitem{CHIU08} D. Covert, D. Hart, A. Iosevich and I. Uriarte-Tuero, An analog of the Furstenberg- Katznelson-Weiss theorem on triangles in sets of positive density in finite field geometries, preprint (2008).

\bibitem{Cr04} E. Croot, {\it Sums of the Form $1/x_1^k+\dots 1/x_n^k$ modulo a prime}, Integers \textbf{4} (2004).

\bibitem{Erd05} M. B. Erdo\u{g}an. {\it A bilinear Fourier extension theorem and applications to the distance set problem.} Internat. Math. Res. Notices  {\bf 23} (2005), 1411--1425.

\bibitem{Fa86} K. J. Falconer {\it On the Hausdorff dimensions of distance sets}, Mathematika \textbf{32} (1986) 206-212.

\bibitem{FKW90} H. Furstenberg, Y. Katznelson, and B. Weiss, {\it Ergodic theory and configurations in sets of positive density} Mathematics of Ramsey theory, 184-198, Algorithms Combin., 5, Springer, Berlin, (1990).


%\bibitem{Ga07} M. Garaev, {\it An explicit sum-product estimate in $\mathbb F_p$}.  Int. Math. Res. Not. IMRN  2007,  no. 11, Art. ID rnm035, 11 pp.

\bibitem{Gl06} A. Glibichuk, {\it Combinatorial properties of sets of residues modulo a prime and the Erd\"os-Graham problem}, Mat. Zametki, \textbf{79} (2006), 384-395; translation in: Math. Notes \textbf{79} (2006), 356-365.

\bibitem{GK08} A. Glibichuk, {\it Additive properties of product sets in an arbitrary finite field}, preprint.

\bibitem{GK06} A. Glibichuk and S. Konyagin. {\it Additive properties of product sets in fields of prime order}. Additive combinatorics,  279--286, CRM Proc. Lecture Notes, 43, Amer. Math. Soc., Providence, RI, 2007. 

\bibitem{HI07} D. Hart and A. Iosevich, {\it Ubiquity of simplices in subsets of vector spaces over finite
fields}, Analysis Mathematica, \textbf{34}, (2007). 

%\bibitem{HI08} D. Hart and A. Iosevich. {\it Sums and products in finite fields: an integral geometric viewpoint}. Contemporary Mathematics, Vol. 464, (2008).

\bibitem{HIKR07} D. Hart, A. Iosevich, D. Koh and M. Rudnev, {\it Averages over hyperplanes, sum-product theory in vector spaces over finite fields and the Erd\H os-Falconer distance conjecture}, Trans. Amer. Math. Soc. (accepted for publication). 

\bibitem{HIKSU08} D. Hart, A. Iosevich, D. Koh, S. Senger and I. Uriarte-Tuero, {\it Distance graphs in vector spaces over finite fields, coloring and pseudo-randomness}, submitted for publication (2008). 

\bibitem{IKo07} A. Iosevich and D. Koh, {\it Extension theorems for the Fourier transforms associated with non-degenerate quadratic surfaces in vector spaces over finite fields}, Illinois Math. J. (accepted for publication), (2007). 

%\bibitem{IK2008} A. Iosevich and D. Koh, {\it The Erd\H os-Falconer distance problem, exponential sums, and Fourier analytic approach to incidence theorems in vector spaces over finite fields}, Siam J. Discrete Math., Vol. 23, No. 1, (2008), 123--135.
\bibitem{IKo08} A. Iosevich and D. Koh, {\it Extension theorems for spheres in the finite field setting}, Forum Math., (accepted for publication), (2008).

\bibitem{IR07} A. Iosevich and M. Rudnev. {\it Erd\H os distance problem in vector spaces over finite fields}.  Trans. Amer. Math. Soc.  \textbf{359}  (2007),  no. 12, 6127--6142. 

\bibitem{KT04} N. Katz and G. Tardos {\it A new entropy inequality for the Erd\"os distance problem} 
Contemp. Math. \textbf{342}, Towards a theory of geometric graphs, 119-126, Amer. Math. Soc., Providence, RI (2004). 

\bibitem{LN97} R.~Lidl and H.~Niederreiter,  {\it Finite fields}, Cambridge University Press, (1993). 

\bibitem{M06} A. Magyar, {\it On distance sets of large sets of integers points},  Israel J. Math. {\bf 164}  (2008), 251--263.

\bibitem{M07} A. Magyar, {\it $k$-point configurations in sets of positive density of ${\Bbb Z}^n$}, Duke Math J. (to appear), (2007).

%\bibitem{MS99} P. Mattila and P. Sjolin, {\it Regularity of distance measures and sets}, Math. Nachr. \textbf{204} (1999), 157-162.

\bibitem{PS00} Y. Peres and W. Schlag, {\it Smoothness of projections, bernoulli convolutions, and the dimension of exceptions}, 
Duke Math J. {\bf 102} (2000), 193-251.

\bibitem{S07} I. Shparlinski, {\it On The Solvability of Bilinear Equations in Finite Fields}. (preprint), (2007)

\bibitem{SV05} J. Solymosi and V. Vu, {\it Near  optimal bounds for the number of distinct distances in high dimensions} Combinatorica (2005). 

\bibitem{TV06} T. Tao and V. Vu, {\it Additive Combinatorics}, Cambridge University Press, (2006).

%\bibitem{W03} T. Wolff {\it Lecture notes on harmonic analysis. With a foreword by C. Fefferman and preface by I. \L aba. Edited by I. \L aba and C. Shubin.} University
%Lecture Series, {\bf 29}. American Mathematical Society, Providence, RI,
%2003.

\bibitem{V08I} L. A. Vinh, {\it On kaleidoscopic pseudo-randomness of finite Euclidean graphs}, preprint (2008). 

\bibitem{V08II} L. A. Vinh, {\it Triangles in vector spaces over finite fields}, preprint (2008). 

\bibitem{W99} T. Wolff, {\it Decay of circular means of Fourier transforms of measures}, Internation Math Research Notices, {\bf 10}, 547-567, (1999). 

\end{thebibliography}
\enddocument